\documentclass[10pt]{article}
\usepackage{graphicx}
\usepackage{amsmath}
\usepackage{amsthm}
\usepackage{amssymb}
\usepackage{epstopdf}
\usepackage{mathrsfs,epsfig}

\usepackage{color}

\usepackage{cite}

\usepackage{mathtools}
\mathtoolsset{showonlyrefs}

\newcommand\bs{\boldsymbol}
\newcommand\ds{\displaystyle}
\def\ee{\varepsilon}
\def\ep{\varepsilon}
\def\eu{{\rm e}}
\def\lit{{\lim_{t \to +\infty}}}
\def\lito{{\lim_{t \to -\infty}}}

\newcommand{\R}{{\mathbb R}}

\newcommand{\K}{{\mathcal K}}
\newcommand{\HH}{{\mathcal H}}
\newcommand{\linea}{{\mathcal L}}

\newcommand{\RDIR}{{\mathcal R}}

\newcommand{\ov}[1]{\overline{#1}}
\newcommand{\und}[1]{\underline{#1}}

\newcommand{\Q}{{\bs Q}}  

\newcommand{\QQ}{{\bs Q}}
\newcommand{\SSS}{{\bs S}}
\newcommand{\QQQ}{{\bs{{\cal Q}}}}
\newcommand{\RR}{{\bs R}}
\newcommand{\EE}{{\bs E}}

\def\Hinc{\boldsymbol{({\rm H}_\uparrow\!\!\!)}}
\def\Hell{\boldsymbol{({\rm H}_\ell)}}
\def\Hstab{\boldsymbol{({\rm W}_s)}}
\def\Hinstab{\boldsymbol{({\rm W}_u)}}
\def\lloc{\textrm{loc}}

\def\la{\lambda}

\newcommand{\fran}[1]{\textcolor{red}{#1}}

\usepackage{mathrsfs}

\begingroup
\newtheorem{theorem}{Theorem}
\newtheorem{lemma}[theorem]{Lemma}
\newtheorem{proposition}[theorem]{Proposition}
\newtheorem{corollary}[theorem]{Corollary}

\newtheorem{remark}[theorem]{Remark}

\endgroup

\title{A bifurcation phenomenon for the critical Laplace and $p$-Laplace equation in the ball}

\author{Francesca Dalbono\thanks{Dipartimento di Matematica e Informatica, Universit\`a degli Studi di Palermo,
Via Archirafi 34, 90123 Palermo - Italy. email: francesca.dalbono@unipa.it
}\\
Matteo Franca\thanks{Dipartimento di   Matematica,
Universit\`a degli Studi di Bologna, Piazza di porta san Donato 5 - 40126, Bologna -
Italy. email: matteo.franca4@unibo.it  }\\
Andrea Sfecci\thanks{Dipartimento di Matematica e Geoscienze,
Universit\`a degli Studi di Trieste, Via Valerio 12/1, 34127, Trieste -
Italy. email: asfecci@units.it.}}

\date{Received: date / Accepted: date}

\begin{document}

\maketitle

\begin{abstract}
In this paper we show that the number of radial positive solutions of the following critical problem
\begin{equation*}
\begin{cases}
\Delta_p u(x) + \la\K(|x|) \,u(x) \, |u(x)|^{q-2} =0\,, \\
\noalign{\smallskip}
u(x)>0 & \quad |x|<1,\\
\noalign{\smallskip}
u(x)=0 & \quad |x|=1,
\end{cases}
\end{equation*}
where  $q= \frac{np}{n-p}$,  $\frac{2n}{n+2} \le p \le 2$ and $x \in \R^n$,
undergoes a bifurcation phenomenon.
Namely, the problem admits one solution for any $\la>0$ if $\K$ is steep enough at $0$, while it admits no solutions
for $\la$ small and two solutions for $\la$ large if $\K$ is too flat at $0$.

The existence of the second solution is new, even in the classical Laplace case.
The proofs use  Fowler transformation and dynamical systems tools.
\end{abstract}

\medskip

\noindent

\textbf{Mathematics Subject Classification (MSC)}: 35J92; Secondary: 35J62, 35B33, 35B09, 34C45.\\
 \textbf{Keywords:}   Scalar curvature equation, bifurcation phenomena, radial solutions, order of flatness,
Fowler transformation,
invariant manifold, phase plane analysis\\

\section{Introduction}

In this paper we focus on  positive radial solutions for the generalized scalar curvature equation
\begin{equation}
\label{plaplace}
\Delta_p u + \K(|x|)\, u|u|^{q-2} =0,
\end{equation}
where $\Delta_p u=\textrm{div}(\nabla u |\nabla u|^{p-2})$ denotes the $p$-Laplace operator,
$x\in\mathbb{R}^n$,
$2\,\frac{n}{2+n}\leq p\leq 2$, \hspace{0.2mm}
 $1<p<n$, and $q$ is the Sobolev critical exponent
\begin{equation}
\label{scalcur}
q=p^*=\frac{np}{n-p}.
\end{equation}
The function  $\K:[0,+\infty[{}\to {}]0,+\infty[{}$ is assumed to be $C^1$, for simplicity.
We are interested in {\em crossing} solutions, which means solutions of the problem
\begin{equation}\label{Problem}
\begin{cases}
\Delta_p u(x) + \K(|x|) \,u(x) \, |u(x)|^{q-2} =0\\
\noalign{\smallskip}
u(x)>0 & \quad |x|<\RDIR\\
\noalign{\smallskip}
u(x)=0 & \quad |x|=\RDIR.
\end{cases}
\end{equation}

In particular, we will show that the existence of radial solutions of~\eqref{Problem} depends on the behavior of
$\K$ in a neighborhood of zero, and on the length of the radius $\RDIR$.

Since we exclusively concentrate our study on radial solutions, we can reduce equation~\eqref{plaplace} to
the following singular ordinary differential equation
\begin{equation} \label{eq.pla}
(r^{n-1} \, u'(r)|u'(r)|^{p-2})'+ r^{n-1}\,\K(r) \, u|u|^{q-2}=0 \,,
\end{equation}
where ``\,$\,'\,$\,'' denotes the differentiation with respect to $r = |x|$, and, with a slight
abuse of notation, $u(r) = u(x)$.

We say that a solution $u(r)$ of~\eqref{eq.pla} is {\em regular} if and only if $\lim_{r \to 0} u(r)=d>0$ for a suitable $d>0$: in this case it will be denoted by $u(r;d)$.
It is well known that $u(r;d)$ exists and it is unique for any $d>0$, cf. e.g.~\cite{FLS,GHMY,KabYY,KNY}.
Note that $u'(0;d)=0$.

\medbreak

As an alternative, from a standard scaling argument, we can find a counterpart for the {\em eigenvalue} equation where we fix $\RDIR=1$ for definiteness and we multiply the potential $\K$ by a parameter $\la$, i.e.
\begin{equation}\label{Problem.la}
\begin{cases}
\Delta_p u(x) + \la\K(|x|) \,u(x) \, |u(x)|^{q-2} =0\\
\noalign{\smallskip}
u(x)>0 & \quad |x|<1\\
\noalign{\smallskip}
u(x)=0 & \quad |x|=1.
\end{cases}
\end{equation}
Indeed, if we set
\begin{equation}\label{change.la}
w(s)=u(r),\qquad s=\frac{r}{\la^{1/p}},
\end{equation}
and $\mathscr{K}(s)=\K(s \la^{1/p})$, then  $u$ solves equation~\eqref{eq.pla} if and only if $w$ solves
$$(s^{n-1}\, w'(s)|w'(s)|^{p-2})'+ \la s^{n-1} \,\mathscr{K}(s) \,w|w|^{q-2}=0,$$
and $w(1)=0$ if and only if $u(\RDIR)=0$, where $\RDIR:=\la^{1/p}$.

Hence, the number of solutions $u$ of~\eqref{Problem} as $\RDIR$ varies coincides with
the number of solutions $u$ of~\eqref{Problem.la} as $\la$ varies.


\medbreak

The scalar curvature equation~\eqref{plaplace} has been extensively studied in the literature due to its significance in a broad variety of applications, such as
Riemannian geometry, astrophysics,  quantum mechanic, chemistry, theory of non-Newtonian fluids and elasticity (cf.~\cite{FrancaCAMQ} for more detailed references on application of $p$-Laplace equations and e.g.~\cite{CL97,CL1}
for application to Riemannian geometry in the $p=2$ case).

 In many phenomena, positivity of solutions has a physical relevance.

Non-linear eigenvalue problems similar to \eqref{Problem.la} are nowadays a classical topic, see e.g. \cite{BrN} and the more recent \cite{BN, DFl,FFna, GW, MP}
where the reaction term $\K(r) u^{q-1}$ is replaced by a sum of a linear and a term either critical or supercritical   with respect to the Sobolev exponent. We address the interested reader to the introduction of \cite{BrN,DFl, FFna} for a discussion of several possible diagrams appearing as different reaction terms are considered.

\medbreak

A key role in our analysis is played by the following hypothesis
\begin{description}
\item[$\Hell$] There are some constants $A,B,\ell>0$ such that
$$
\K(r)= A + B r^{\ell} + h(r) \qquad\mbox{and}\qquad
\lim_{r\to 0}  \frac{|h(r)|+r|h'(r)|}{r^{\ell}}=0\,.
$$
\end{description}
The existence and the multiplicity of the solutions of~\eqref{Problem} and~\eqref{Problem.la}
depend  crucially  on $\ell$, which, in literature, is often referred to as the \emph{order of flatness} of the function $\K$ at $r=0$.

Problem~\eqref{Problem} subject to  condition $\Hell$ has been already investigated in the early 1990s
by Bianchi-Egnell in~\cite{BE1} and by Lin-Lin in~\cite{LL}, who determined the existence of the critical value
 $\ell^*_{2}=n-2$
in the Laplacian setting $p=2$. In fact, Bianchi-Egnell and Lin-Lin have been able to prove the following result.

\smallskip

\noindent {\bf Theorem A~\cite{LL}.}\,
{\em Assume that $\K$ satisfies $\Hell$ and consider $\mathbb{\bf p=2}$ in~\eqref{Problem}.
\begin{itemize}
\item[$(i)$]
If $\ell<n-2$, then problem~\eqref{Problem} admits a radial solution for every $\RDIR>0$.
\item[($ii)$]
 If $\ell \geq n-2$, then there exists a sufficiently small constant $r_0>0$ such that problem~\eqref{Problem}
  does not admit any radial solution when $\RDIR<r_0$;
\item[$(iii$)]
If $n-2 \leq \ell<n$,  then there exists a sufficiently large constant ${R}_0>0$ such that problem~\eqref{Problem} admits a radial solution
 for every $\RDIR\geq {R}_0$.
\end{itemize}
}

In fact Bianchi and Egnell in~\cite{BE1} focused on the $R=1$ case.
In particular, following a shooting approach based on ordinary differential equations, they constructed and glued together two regular solutions of~\eqref{eq.pla}, one shooting from the zero initial condition and the other shooting from infinity.

The   restriction $\ell < n$ has often been adopted in the Laplacian literature to ensure the existence of a solution to problem~\eqref{Problem}.
For instance, it appears in~\cite[Theorem 0.19]{Lsolo}, where the scalar curvature $\K$ is required to be strictly decreasing in a left neighbourhood of $\RDIR$.
The upper constraint $\ell<n$ can be also found in the recent work~\cite{LNW}, dealing with the $\sigma_k$-Nirenberg problem on the standard sphere
$\mathbb{S}^n$ for $2 \leq k < n/2$.
Among the very few examples of existence results for the Laplacian scalar curvature problem
 in the absence of upper bound conditions on $\ell$,
we refer to the very recent papers~\cite{CHS,Sharaf}, where hypothesis $\Hell$  is combined with
 a (not so easily verifiable) topological global index formula on the critical points of $\K$.

 Note that in the $p$-Laplace context, the condition $\ell<n$  generalizes to $\ell <\frac{n}{p-1}$,  cf.~\cite{FrancaFE}.

We emphasize that Theorem A, besides its intrinsic interest, has been a key starting point in proving
the existence of Ground States with fast decay, i.e.  solutions $u(r)$ positive for any $r>0$ and decaying as $r^{-(n-2)}$
at infinity.

In fact, it can be shown that if $\K$ is increasing close to $r=0$ and decreasing close to $r=\infty$
we might expect to find Ground States with fast decay: in the 90s  there was a flourishing of papers giving sufficient conditions for existence and non-existence
of these solutions.
A possible strategy is indeed to combine Theorem A, or similar results, with the use of Kelvin inversion,
which transfers the information on regular solutions to fast decay solutions,
 see e.g.~\cite{BE1,DF}.
 Roughly speaking, one can expect that if $\K$ is steep enough at $0$ (i.e. $\ell<n$) and at infinity
  $\K (r) \sim a+br^{-{\ell}}$
 with $\ell<n$ and $a>0$, $b>0$, then
there is a Ground State  with fast decay, while if these conditions are violated, one can construct a counterexample to Theorem A, see~\cite[Theorem 0.3]{BE1}. For a generalization to the $p$-Laplace setting, we also refer to~\cite{FrancaFE}, which
follows a different strategy since Kelvin inversion is not available in that context.

 We think it is worthwhile to point out that if $A=0$ in $\Hell$, then the problem becomes easier: roughly speaking,
 its solutions behave as the solutions of the $A>0$, $\ell$-subcritical case,
 i.e.~\eqref{Problem} admits a radial solution for any $\RDIR>0$ (or, equivalently,~\eqref{Problem.la} admits
a radial solution for any $\la>0$),  for any $\ell>0$. Again, combining this result with Kelvin inversion one might obtain
the existence  of Ground States with fast decay. This idea was extensively used in the 90s in many papers
to handle the Laplacian problem, starting, probably, from~\cite{KSY},
\cite{ChCh},~\cite{YY93},
and then it was developed and adapted to related problems, such as existence of Ground States with fast decay with a prescribed number of sign changes, see e.g.~\cite{YY94} and~\cite{KabYY}, dealing with the Laplacian and $p$-Laplacian setting, respectively.

\medbreak

The aim of this paper is to improve Theorem A in 3 main directions.

Firstly, we extend  the results to the $p$-Laplace context, proving the existence of the generalized flatness order's threshold
$$
\ell^*_{p}=\frac{n-p}{p-1}\,,
$$
which coincides with $\ell^*_{2}=n-2$
of Theorem A for $p=2$.
As far as we are aware, these are the first results in this direction,
 although
 we have to require
the probably technical condition
$\frac{2n}{2+n} \le p \le 2$. The possibility to remove this restriction will be the object of further investigations.

Secondly, we are able to remove the condition $\ell \le  n$ (i.e. $\ell \le \frac{n}{p-1}$ in the $p$-Laplace context), by requiring
that $\K(r)$ is increasing for any $r>0$.

Thirdly, and probably more importantly, we are able to prove the existence of a second solution when $\ell> n-2$
(i.e. $\ell>\ell^*_{p}$ in the $p$-Laplace context), thus completing the bifurcation diagram even in the $p=2$ case.

\vspace{1mm}

Let us state our assumptions and the main results of the paper.
\begin{description}
\item[$\Hinc$]
 The function $\K(r)$ is increasing for any $r \ge 0$, strictly in some interval.
  \item[$\Hstab$]  The function $\K(\eu^t)$ is  uniformly continuous in $[0,+\infty[{}$ and there are $\ov{K}>\und{K}>0$ such that
  $\und{K}<\K(r)<\ov{K}$ for any $r>0$.
\end{description}

\noindent
Note that  if $r\K'(r)$ is bounded in
$[1,+\infty[$ then $\K(\eu^t)$ is uniformly  continuous in $[0,+\infty[{}$.

\begin{theorem}\label{t.main}
Assume that $\K$ satisfies $\Hell$.
\begin{itemize}
\item[] If $\ell <  \ell^*_{p}$, then problem~\eqref{Problem} admits a radial solution for every $\RDIR>0$.
\item[] If $\ell  \ge   \ell^*_{p}$, then there exists $\RDIR_0>0$ such that problem~\eqref{Problem}
   does not admit any radial solution when $\RDIR < \RDIR_0$.
\end{itemize}
\end{theorem}

\begin{theorem}\label{t.main.due}
Assume that $\K$ satisfies $\Hell$ and $\Hinc$.\\ \\
If $\ell = \ell^*_{p}$, then there exists $\RDIR_0>0$ such that problem~\eqref{Problem}
\begin{itemize}
    \item  does not admit any radial solution when $\RDIR < \RDIR_0$;
    \item   admits at least a radial solution when     $\RDIR> \RDIR_0$.
\end{itemize}
If $\ell >  \ell^*_{p}$, then there exists $\RDIR_0>0$ such that problem~\eqref{Problem}
\begin{itemize}
    \item  does not admit any radial solution when $\RDIR<\RDIR_0$;
    \item   admits at least a radial solution when  $\RDIR= \RDIR_0$;
     \item   admits at least $2$ radial solutions when  $\RDIR> \RDIR_0$.
\end{itemize}
\end{theorem}

In fact, Theorem~\ref{t.main.due} holds also if we drop the global assumption $\Hinc$,
but we strengthen the requirement on $\K'(0)$ by
 introducing the  upper bound on the order of flatness $\ell\leq \frac{n}{p-1}$ at zero, in the spirit of Theorem A.
 Unfortunately, we need to ask for some further  very weak technical conditions on $\K$ for $r$ large.
\begin{theorem}\label{t.main.tre}
      Assume that $\K$ satisfies $\Hell$ with
 $\ell^*_{p} \le  \ell \le \frac{n}{p-1}$.
Assume further that either $\Hstab$ holds or there is $\rho_1>0$ such that $\K(r) \ge \K(\rho_1)$
      when $r \ge \rho_1$, then we get the same conclusion as in Theorem  $\ref{t.main.due}$.
    \end{theorem}

    Using the change of variable~\eqref{change.la}, we can
    rewrite Theorems~\ref{t.main},~\ref{t.main.due}, and~\ref{t.main.tre} as follows.

\begin{corollary}\label{c.main}
Assume that $\K$ satisfies $\Hell$.
\begin{itemize}
\item[] If $\ell <   \ell^*_{p}$, then problem~\eqref{Problem.la} admits a radial solution for every $\la>0$.
\item[] If $\ell  \ge  \ell^*_{p}$, then there exists $\la_0>0$ such that problem~\eqref{Problem.la}
  does not admit any radial solution when $\la < \la_0$.
\end{itemize}
\end{corollary}
\begin{corollary}\label{c.main.due}
Assume that $\K$ satisfies $\Hell$ and $\Hinc$. \medskip

\noindent
If $\ell =  \ell^*_{p}$, then there exists $\la_0>0$ such that problem~\eqref{Problem.la}
\begin{itemize}
    \item  does not admit any radial solution when $\la < \la_0$;
    \item   admits at least a radial solution when  $\la > \la_0$.
\end{itemize}
If $\ell > \ell^*_{p}$, then there exists $\la_0>0$ such that problem~\eqref{Problem.la}
\begin{itemize}
    \item  does not admit any radial solution when $\la < \la_0$;
    \item   admits at least a radial solution when  $\la= \la_0$;
    \item   admits at least two radial solutions when  $\la> \la_0$.
\end{itemize}
\end{corollary}

  Again, according to Theorem~\ref{t.main.tre}, we can drop the global condition $\Hinc$ and get the same result by restricting the interval in which $\ell$ varies
  and by imposing some further weak asymptotic conditions on $\K(r)$ for $r$ large.
 \begin{corollary}\label{c.main.tre}
Assume that $\K$ satisfies $\Hell$ with $ \ell^*_{p} \le   \ell \le \frac{n}{p-1}$; assume further that either $\Hstab$ holds or there is $\rho_1>0$ such that
$\K(r) \ge \K(\rho_1)$ for any $r \ge \rho_1$.
Then, we get the same conclusions as in Corollary $\ref{c.main.due}$.
\end{corollary}

In fact via Theorem~\ref{t.main.tre} and Corollary~\ref{c.main.tre}
we are also able to extend Theorem A to the case where $\ell = \frac{n}{p-1}$, which was not covered by~\cite{BE1,LL}, with the addition of a very weak technical condition
which, roughly speaking, is satisfied unless $\K(r)$ is subject to wild oscillations  for $r$ large, or converges to $0$ for $r$ large.

Moreover, our methods are considerably different from those of Lin-Lin in~\cite{LL}.
Our proofs are based on the Fowler transformation, which discloses the geometrical aspect of our problem by converting the singular ordinary differential equation~\eqref{eq.pla}
into an equivalent dynamical system, cf.~\eqref{sist-din}, following the way paved by~\cite{JK, JPY, JPY2} and later on by~\cite{BFdP, BJ, FJ}.
 Then, we develop   a detailed phase plane analysis, involving  invariant manifold theory  for non-autonomous systems, energy estimates, comparisons of the
 non-autonomous planar system with suitable autonomous ones and a Gr\"{o}nwall's argument. 

\subsection{On the proofs of the theorems}
\label{sec-sunto}

We briefly sketch the plan of our proof.
Let
\begin{equation}\label{defJ}
  J= \{ d>0 \mid u(r;d) \; \textrm{ is a crossing solution} \},
\end{equation}
and denote by $R(d)$ the {\em  first zero} of $u(r;d)$ when $d \in J$:
our argument relies on a study of the properties of the function $R: J \to {}]0,+\infty[{}$.

Using some classical results (see~\cite{DN,KN} for the Laplacian
case, and~\cite{GHMY,KNY} for $p$-Laplacian extensions),
we know  that $J=\fran{}]0,+\infty[{}$ under assumption $\Hinc$.
Then, using a standard transversality argument,
it is straightforwardly proved that $R(d)$ is continuous, see Proposition~\ref{p.Rcontinua} below.

Then, it is not difficult to show, see, e.g.~\cite[Proposition 2.4]{KabYY}   that
\begin{equation}
\label{alt0}
    \lim_{d \to 0} R(d)=+\infty\,.
\end{equation}

\begin{figure}[h]
\centerline{
\epsfig{file=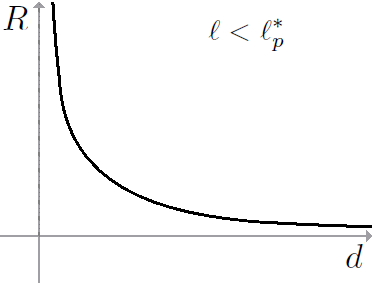,  height = 30 mm}
\qquad\qquad
\epsfig{file=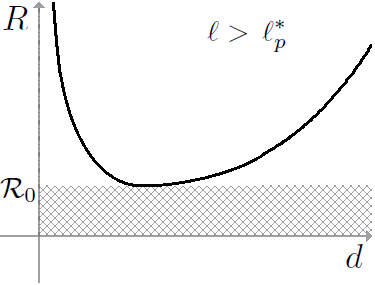,  height = 30 mm}
}
\vspace{5mm}
\centerline{
\epsfig{file=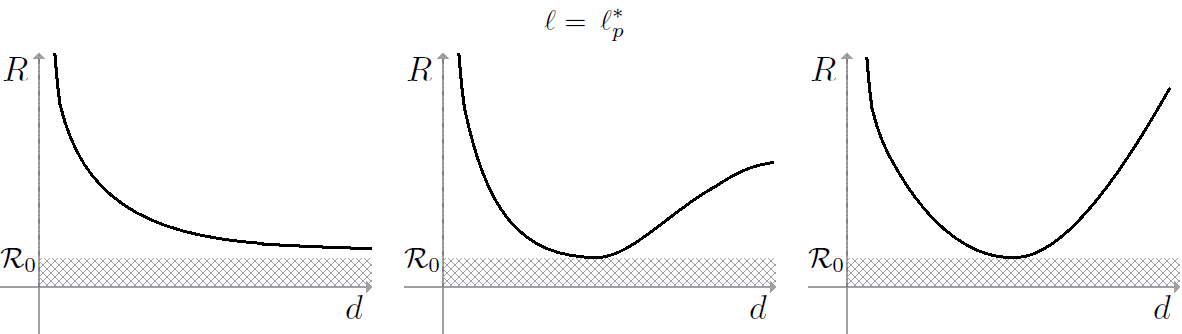, height = 34 mm}
}
\caption{A sketch of the graph of the function $R(d)$, when $\Hell$ and $\Hinc$ are assumed. In the critical case $\ell=\ell^*_{p}$, we have three possible alternatives.}
\label{diagram.inc}
\end{figure}

The focus and the original part of our study consists in analyzing the asymptotic behaviour of the solutions with large initial data $d$,
which is determined by the parameter $\ell$ in $\Hell$. In particular, we have the following
bifurcation phenomenon, when $\Hinc$ is assumed
\begin{equation}
\label{altinf}
\begin{cases}
    \ds\lim_{d \to +\infty} R(d) =0 \,,
    & \mbox{ if }\displaystyle{0<\ell< \ell^*_{p}}\,,\\
    \ds\liminf_{d \to +\infty} R(d) >0 \,,
    & \mbox{ if }\displaystyle{\ell= \ell^*_{p}}\,,\\
    \ds\lim_{d \to +\infty} R(d) =+\infty \,,
    & \mbox{ if }\displaystyle{\ell>  \ell^*_{p}}\,,\\
\end{cases}
\end{equation}
see Propositions~\ref{p.sub-bis} and~\ref{p.parziale} below.
This asymptotic analysis will allow us to draw the diagrams in Figure~\ref{diagram.inc}.

Since the number of solutions of
problem~\eqref{Problem} is the number of points of the preimage $R^{-1}(\RDIR)$, the proof of  Theorem~\ref{t.main.due} is immediately given.

Actually, using a truncation argument, see Remark~\ref{r.k}, we are
 able to get the first estimate in~\eqref{altinf}, also when assumption $\Hinc$ is dropped,
cf. Proposition~\ref{p.sub-bis}. Hence, the proof of Theorem~\ref{t.main} in the subcritical case follows.

In fact, in the $p=2$ case,  a part of~\eqref{altinf}
has been already shown in~\cite[Theorem 1.6]{LL}, see also~\cite[Remark 4.1]{CL97}:
$$
\begin{cases}
\displaystyle \lim_{d \to +\infty} R(d)=0 & \mbox{\, if  \,}\ell<\ell^*_2 \,,\\
\noalign{\smallskip}
\exists \, r_0>0 \,: \,  R(d)\geq r_0
\,,\, \forall\, d>0
&\mbox{\, if  \,}\ell\geq\ell^*_2 .
\end{cases}
$$

\begin{figure}[h]
\centerline{
\epsfig{file=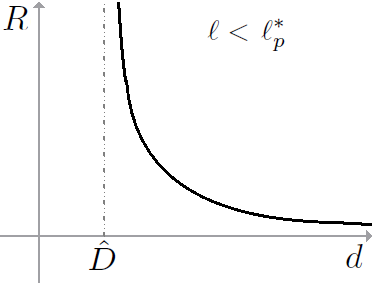,  height = 30 mm}
\qquad\qquad
\epsfig{file=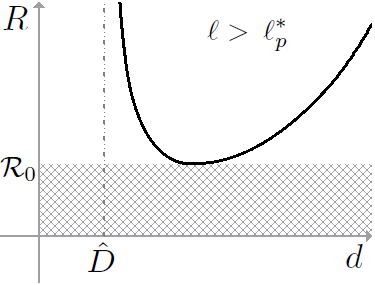,  height = 30 mm}
}
\vspace{5mm}
\centerline{
\epsfig{file=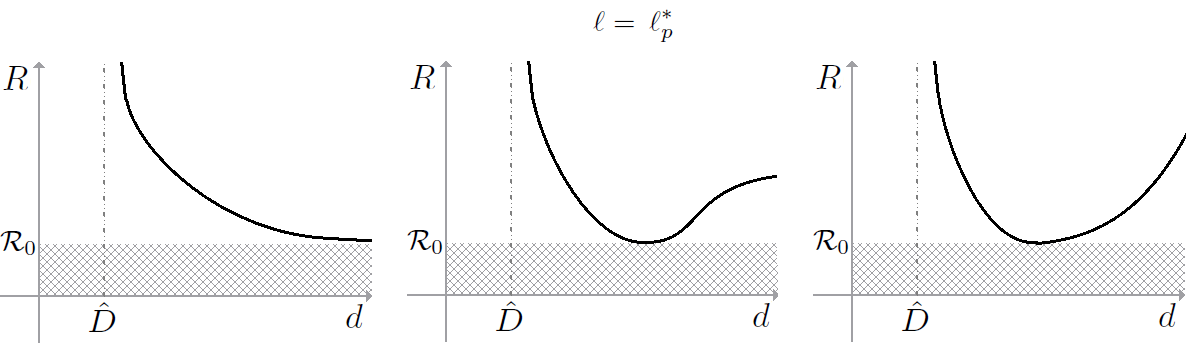, height = 34 mm}
}
\caption{A sketch of the graph of the function $R(d)$, in the setting of Theorem~\ref{t.main.tre}. In the critical case $\ell=\ell^*_{p}$, we have three possible alternatives.}
\label{diagram.noninc}
\end{figure}

\medbreak

Finally, we  adapt our analysis to the context where the global requirement $\Hinc$ is replaced by the local requirement $\ell \le \frac{n}{p-1}$,
and we reprove   the existence of the second solution in the supercritical case. As shown in Lemma~\ref{l.new.technical} below, the restriction $\ell \le \frac{n}{p-1}$, and the technical
requirement that either $\Hstab$ holds or there is $\rho_1>0$ such that $\K(r) \ge \K(\rho_1)$ when $r \ge \rho_1$
are needed just
in order to ensure that there is $\hat D \ge 0$ such that ${}]\hat D,+\infty[{} \subset J$, with $J$ as in~\eqref{defJ}.

Then, using classical arguments, we see that if $\hat D \not\in J$, then
$$\lim_{d \to \hat D} R(d)=+\infty,$$
and, adapting the computation performed when $\Hinc$ holds, we   prove  that
$R(d)$ satisfies~\eqref{altinf}  if $ \ell \le \frac{n}{p-1}$,
see Propositions~\ref{p.nuova} and~\ref{p.parzialebis} below.
As a consequence, we are able to draw the diagrams for $R(d)$  in Figure~\ref{diagram.noninc},
 from which Theorem~\ref{t.main.tre} follows.

Notice that  if $\hat D>0$, then $u(r; \hat D)$ is a Ground State, that is a solution of~\eqref{eq.pla}, positive for any $r>0$.

The paper is organized as follows.
In \S\ref{basic} we introduce the Fowler transformation, i.e.
the change of variables~\eqref{transf1} which turns~\eqref{eq.pla} into the planar, non-autonomous dynamical system~\eqref{sist-din}; then we recall some basic tools
in this context. In particular, we review some aspects of invariant manifold theory for non-autonomous systems,
and we  introduce the unstable leaves $W^u(\tau)$ of~\eqref{sist-din} which correspond to regular solutions
of~\eqref{eq.pla}.
In \S\ref{basicRd} we establish some standard properties of the function $R(d)$, such as continuity and
asymptotic properties close to $d=0$ and to $d= \hat{D}$, when $u(r;\hat D)$ is a Ground State.
The core of our argument is in \S 4, where we prove the asymptotic properties of $R(d)$ as $d$ tends to infinity:
in \S\ref{S.sub-new} we focus on the subcritical case $\ell< \ell^*_{p}$, and in \S\ref{S.super} we consider the critical and supercritical setting
$\ell \ge \ell^*_{p}$. In \S 5 we briefly conclude the proofs of the main results of the paper.

\section{Fowler transformation,\\ basic notation and preliminaries}\label{basic}

Let us  introduce a change of variables known as Fowler transformation,
which allows to transform (\ref{eq.pla}) into a two-dimensional dynamical system.
We define
\begin{align}
\label{transf1}
 & \alpha=\displaystyle{\frac{n-p}{p}}, \qquad  \beta= \displaystyle{\frac{n(p-1)}{p}}, \nonumber\\
\noalign{\medskip}
 & x=u(r)r^{\alpha}\qquad  y=u'(r)|u'(r)|^{p-2}r^{\beta} \qquad
r=\eu^t.
\end{align}
This change of variable is known from the '30s, see~\cite{Fow}, and it has been generalized
to the $p$-Laplacian case by Bidaut-V\'eron~\cite{Bv} and some years later (independently) by Franca, see e.g.~\cite{FrancaTMNA,Fcorri,FrancaAM,FrancaFE,FJ}.\\
According to (\ref{transf1}), we can rewrite (\ref{eq.pla}) as the following dynamical system:
\begin{equation}
\label{sist-din}
\left( \begin{array}{c}
\dot{x}\\
\noalign{\medskip}
\dot{y}
\end{array}\right)
=
\left( \begin{array}{cc}
\alpha& 0  \\
\noalign{\medskip}
 0 & -\alpha
\end{array} \right)
\left(
 \begin{array}{c} x \\
\noalign{\medskip}
y
\end{array}\right)
+
\left(\begin{array}{c}
y\, |y|^\frac{2-p}{p-1} \\
\noalign{\medskip}
- K(t) \, x|x|^{q-2}
 \end{array}\right),
\end{equation}
where $K(t)=\K(\eu^t)$,  $q=p^*=\frac{np}{n-p}$ as in~\eqref{scalcur},
 and ``$\cdot$"  denotes the differentiation with respect to $t$.\\
\begin{remark}
\label{r.smooth}
Note that system~\eqref{sist-din} is $C^1$ if and only if $\frac{2n}{2+n}\leq p \leq 2$.
\end{remark}

Given the initial data $\tau\in\R$ and $\QQ\in \R^2$, we will denote by
\begin{equation}\label{flux-sist-din}
\bs{\phi}(t;\tau,\QQ)
=\big(x(t;\tau,\QQ),y(t;\tau,\QQ)\big)
\end{equation}
the trajectory of~\eqref{sist-din} such that $\bs{\phi}(\tau;\tau,\QQ)=\QQ$.

 \noindent Define the energy function
\begin{equation}
\label{Hq}
 \HH(x,y;t):=\alpha xy \, +\, \frac{p-1}{p} \, |y|^{\frac{p}{p-1}}\, +\, K(t) \, \frac{|x|^q}{q}\,.
\end{equation}
If we evaluate $\mathcal H$ along a solution $(x(t),y(t))$ of~\eqref{sist-din}, we obtain the associated Pohozaev type energy $\mathcal H(x(t),y(t);t)$, whose derivative
with respect to $t$ satisfies
\begin{equation}
    \label{Hder}
\frac{d}{dt} \HH\big(x(t),y(t);t\big) = \dot K(t)\, 
\frac{|x(t)|^q}{q}\,.
\end{equation}

We also need to consider the autonomous system obtained by {\em freezing} the $t$-dependence
of $K(\cdot)$. In particular,
fixed $\tau\in\R\cup\{\pm\infty\}$,
we consider the {\em frozen} autonomous system:
\begin{equation}
\label{sist-frozen}
\left( \begin{array}{c}
\dot{x}\\
\noalign{\medskip}
\dot{y}
\end{array}\right)
=
\left( \begin{array}{cc}
\alpha& 0  \\
\noalign{\medskip}
 0 & -\alpha
\end{array} \right)
\left(
 \begin{array}{c} x \\
\noalign{\medskip}
y
\end{array}\right)
+
\left(\begin{array}{c}
y\, |y|^\frac{2-p}{p-1} \\
\noalign{\medskip}
- K(\tau) \, x|x|^{q-2}
 \end{array}\right) ,
\end{equation}
where $\K$ is assumed to have finite limit, whenever $\tau=\pm\infty$.\\

\vspace{1mm}

If we differentiate  the energy $\mathcal H(\cdot;\tau)$ of system~\eqref{sist-frozen} along a solution
\linebreak
$(x(t),y(t))$ of the dynamical system~\eqref{sist-din}, we get
\begin{equation}
    \label{ripristi}
\frac{d}{dt} \HH\big(x(t),y(t);\tau\big) = \big(K(\tau)-K(t)\big) \, x(t)|x(t)|^{q-2}\,\dot{x}(t)\,.
\end{equation}
Let us fix $\tau  \in \R\cup \{\pm\infty\}$ in~\eqref{sist-frozen},
let $\tau_0 \in \R$ and   $\QQ\in\R^2$; we   denote by
\begin{equation}\label{flux-sist-frozen}
\bs{\phi_\tau}(t;\tau_0,\QQ)
=\big(x_\tau(t;\tau_0,\QQ),y_\tau(t;\tau_0,\QQ)\big)
\end{equation}
the trajectory of~\eqref{sist-frozen}, such that  $\bs{\phi_\tau}(\tau_0;\tau_0,\QQ)=\QQ$.

System~\eqref{sist-frozen} exhibits an equilibrium point in $\{x> 0\,, y< 0\}$, of coordinates
\begin{equation}\label{EE}
\EE(\tau) =(E_x(\tau), E_y(\tau))
=
\left(
\left(\frac{\alpha^p}{K(\tau)}\right)^{\frac{1}{q-p}}
\,,\,
- \left( \frac{\alpha^q}{K(\tau)} \right)^{\frac{p-1}{q-p}}
\right).
\end{equation}
It is also well known (see, among others,~\cite{FrancaFE}) that,  for any fixed $\tau\in\R\cup\{\pm\infty\}$, system~\eqref{sist-frozen} admits a homoclinic orbit
(see Figure~\ref{fig-D})
\begin{equation}
\label{gammatau}
\bs{\Gamma_\tau} = \{(x,y) \mid \HH(x,y;\tau)=0\,, x>0 \}	\,,
\end{equation}
recalling that the case $\tau=+\infty$ requires a boundedness restriction on $\K$.
Taking into account that the flows of~\eqref{sist-din} and~\eqref{sist-frozen} are ruled by their linear part near the origin,
 we easily deduce that the origin is a
 saddle-type critical point for~\eqref{sist-din} and~\eqref{sist-frozen}.
Finally, notice that
 $\EE(\tau)$  is a centre and it  lies in the interior of  the region enclosed by $\bs{\Gamma_{\tau}}$.

According to~\cite{FJ,Gaz}, we also know the exact expression of the homoclinic trajectories $\bs{\phi^*}=(\bs{x^*},\bs{y^*})$ of~\eqref{sist-frozen}. In particular, the one corresponding to the regular Ground State $u^*$ such that $u^*(0)=d$ satisfies
\begin{equation}
\label{homosapiens}
x^*(t)=d  \left[ \eu^{-t}+C(d) \, \eu^{\frac{t}{p-1}}   \right] ^{-\frac{n-p}{p}}\,,
\end{equation}
where    $C(d)>0$ is a computable constant, see~\cite{Gaz}.\\
 Moreover,  fixed $\bs{P}\in \bs{\Gamma_{-\infty}}$ there exists a constant $c_{\bs{P}}>0$ such that
\begin{equation}
\label{nonfiniscemai}
\sup\{\|\bs{\phi_{-\infty}}(\theta+ \tau;\tau , \bs{Q})\| \eu^{-\alpha \theta}         \mid \theta \le 0 ,\;  \tau \le 0,\; \bs{Q} \in
\bs{\widetilde{\Gamma}_{-\infty}^{P}}\} \le c_{\bs{P}},
\end{equation}
where $\bs{\widetilde{\Gamma}_{-\infty}^{P}}:=\bs{\phi_{-\infty}}({}]-\infty,0];0 , \bs{P})$.

\medskip

Let us now list some immediate consequences of assumption $\Hinc$.
\begin{remark}
\label{r.incl}
Assume $K(\tau_2)>K(\tau_1)$, then  the homoclinic orbit $\bs{\Gamma_{\tau_2}}$ belongs to the region enclosed by $\bs{\Gamma_{\tau_1}}$, cf. Figure~\ref{fig-D}.
\end{remark}

\begin{figure}[t]
    \centering
    \includegraphics[scale=0.6]{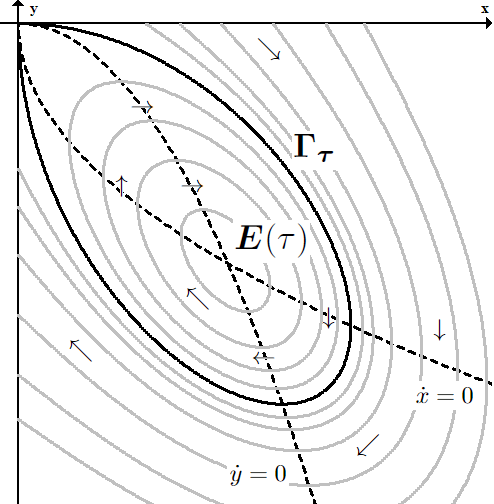}
    \qquad
    \includegraphics[scale=0.6]{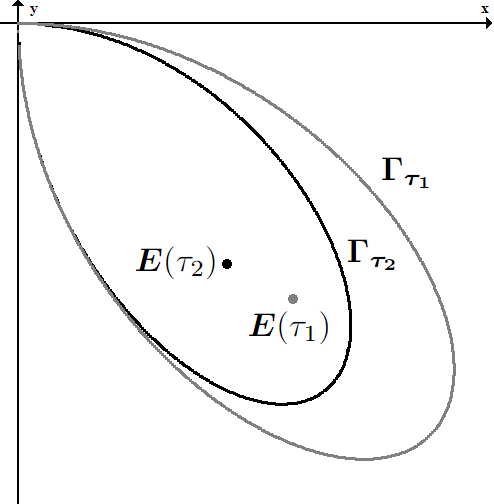}
    \caption{On the left, the energy levels of $\HH$ at fixed time $\tau$, with the homoclinic orbit $\bs{\Gamma_{\tau}}$. On the right, the position of the set $\bs{ \Gamma_{\tau_2}}$ with respect to the set $\bs{\Gamma_{\tau_1}}$ in the case $K(\tau_2)>K(\tau_1)$.}
    \label{fig-D}
\end{figure}

\begin{lemma}
\label{l.basic}
Assume $\Hinc$.
If $u(r;d)$ is a regular solution, then the corresponding trajectory $\bs{\phi}(\cdot;d)=(x(\cdot;d),y(\cdot;d))$ of  system~\eqref{sist-din}
satisfies
$$ \lim_{t\to-\infty} \HH(\bs{\phi}(t;d);t)=0\,,
$$
and
$$
\HH(\bs{\phi}(t;d);t) \geq 0
\quad\mbox{ for every }  t \leq T(d)\,
$$
(the equality occurs only when $K(t)=K(-\infty)$ holds),
where $T(d):= \sup \{ \tau \in \R:\,   x(t;d)>0 \,, \forall t\leq \tau\}$ is the first  zero of $x(\cdot;d)$.
Moreover,
$$
\HH(\bs{\phi}(t;d);\tau)
\geq
\HH(\bs{\phi}(t;d);t)
\geq
0 \quad  \mbox{ for every }  t< \min\{\tau, T(d)\}
$$
(the first equality occurs only when $K(t)=K(\tau)$ holds).
\end{lemma}
\begin{proof}
  The result immediately follows from $\Hinc$,~\eqref{Hder}, and~\eqref{Hq}.
\end{proof}

As an immediate consequence we have the following.

\begin{lemma}
\label{l.ingamma0}
Assume $\Hinc$ and fix $d>0$. Then, for every $t\le T(d) $, the trajectory $\bs{\phi}(t;d)$ of system~\eqref{sist-din}
either lies in the exterior of  the region enclosed by $\bs{\Gamma_{t}}$, or it lies on $\bs{\Gamma_{t}}$ if $K(t)=K(-\infty)$.
\end{lemma}

The next part of the section is devoted to explore the dynamical system~\eqref{sist-din}  through an invariant manifold approach.
From~\cite{FJ,JSell,JPY}, and~\cite[\S 13.4]{CoLe}, we know that the existence of the unstable manifold is ensured by the following condition:
\begin{description}
\item[$\Hinstab$]
The function $K$ is bounded and uniformly continuous in ${}]-\infty,0]$ and $\displaystyle{\lito K(t)=K(-\infty) \ge 0}$ is finite.
  \end{description}
Notice that $\Hinstab$ is always satisfied when $\Hell$ holds.

\smallbreak

Correspondingly, the existence of the stable manifold is  guaranteed by $\Hstab$.
   The unstable manifold will play a key role in this paper, since we will see via Remark~\ref{r.corresp}
  that regular solutions $u(r;d)$ correspond to trajectories of~\eqref{sist-din} converging to the origin as $t \to -\infty$,
  i.e. leaving from the unstable manifold.

Let $B(\bs{0},\delta)$
denote the open ball of radius $\delta$, centered at the origin.
 Following~\cite[Theorem 2.1]{JPY}, which is in fact a rewording of~\cite[Theorem 2.25]{JSell}, and its reformulation in\cite[Appendix]{FrSf2},
we find the next result.
\begin{theorem}\label{local.mani}
Assume $\Hinstab$, then there is $\delta>0$ such that for any $\tau \le 0$ the set
$$M^u_{\lloc}(\tau):= \{ \Q \mid \bs{\phi}(t; \tau, \Q) \in B( \bs{0}, \delta) \; \textrm{for any $t \le \tau$} \}$$
is a $C^1$ embedded manifold tangent to the $x$-axis at the  origin if $\frac{2n}{2+n} < p\leq 2$, and to the line  $y=-\frac{K(-\infty)}{n}\,x$ if $p=\frac{2n}{2+n}$.
Further, let $\linea$
be a segment
transversal to  $M^u_{\lloc}(\tau)$, then $M^u_{\lloc}(\tau) \cap \linea= \{\bs{Q^u}(\tau)\}$ is a singleton,
 and $\bs{Q^u}(\tau)$
is uniformly continuous in $\tau$.

  Assume $\Hstab$, then there is $\delta>0$ such that for any $\tau \ge 0$ the set
$$M^s_{\lloc}(\tau):= \{ \Q \mid \bs{\phi}(t; \tau, \Q) \in B( \bs{0}, \delta) \; \textrm{for any $t \ge \tau$} \}$$
is a $C^1$ embedded manifold tangent to the $y$-axis at the  origin if $\frac{2n}{2+n} \le p<2$ and to the line $y=-(n-2)x$ if $p=2$.
 Further, let $\linea$ be a segment
transversal to  $M^s_{\lloc}(\tau)$, then $M^s_{\lloc}(\tau) \cap \linea= \{\bs{Q^s}(\tau)\}$ is a singleton, and $\bs{Q^s}(\tau)$
is uniformly continuous in $\tau$.
\end{theorem}
This result may also be obtained by using the simpler approach developed in~\cite[\S 13.4]{CoLe}, with the exception of
 the part concerning
the uniform continuity of $\bs{Q^{u}}(\cdot)$ and $\bs{Q^{s}}(\cdot)$.
Actually,~\cite[Theorem 2.1]{JPY} ensures that $\bs{Q^{u}}(\tau)$ and $\bs{Q^{s}}(\tau)$ are $C^1$
 if  $K(t)$ is $C^1$.

We notice that  $M^u_{\lloc}(\tau)$
is split by the origin in two connected components.
Since we are just interested in positive solutions, we denote by $W^u_{\lloc}(\tau)$ and $W^{u,-}_{\lloc}(\tau)$ the components
lying respectively in $x >0$ and $x < 0$. Similarly,  $M^s_{\lloc}(\tau)$
is split by the origin in two connected components, say $W^s_{\lloc}(\tau)$ and $W^{s,-}_{\lloc}(\tau)$,
lying in $x >0$ and $x < 0$, respectively.\\
Following again  either~\cite[Theorem 2.25]{JSell}
 or~\cite[\S 13.4]{CoLe} and in particular Theorem 4.5, we deduce  some crucial properties concerning the uniform exponential asymptotic behaviour of the trajectories intersecting the manifolds $W^u_{\lloc}(\tau)$ and $W^s_{\lloc}(\tau)$.
\begin{theorem}\label{local.mani1}
Assume $\Hinstab$ and $\Hstab$, then the  sets
$$W^u_{\lloc}(\tau_1):= \{ \Q \in M^u_{\lloc}(\tau_1) \mid x(t; \tau_1 , \Q) >0 \; \textrm{for any $t \le \tau_1$} \} \cup \{(0,0) \},$$
$$W^s_{\lloc}(\tau_2):= \{ \Q \in M^s_{\lloc}(\tau_2) \mid x(t; \tau_2 , \Q) >0 \; \textrm{for any $t \ge \tau_2$} \}\cup \{(0,0) \}$$
are $C^1$ embedded manifold for any $\tau_1 \le 0 \le \tau_2$. Furthermore,
if $\bs{Q^u} \in W^u_{\lloc}(\tau_1)$ and $\bs{Q^s} \in W^s_{\lloc}(\tau_2)$, then
the limits
$$\lito \|\bs{\phi}(t;\tau_1, \bs{Q^u}) \| \eu^{-\alpha t} \,, \qquad \lit \|\bs{\phi}(t;\tau_2, \bs{Q^s}) \| \eu^{\alpha t}$$
are positive and  finite. Further, there is $c_0=c_0(\delta)$    such that
$$ \sup\{\|\bs{\phi}(\theta+ \tau_1;\tau_1, \bs{Q^u})\| \eu^{-\alpha \theta}         \mid \theta \le 0 ,\;  \tau_1 \le 0,\; \bs{Q^u} \in W^u_{\lloc}(\tau_1)\} \le c_0(\delta),$$
$$ \sup\{\|\bs{\phi}(\theta+ \tau_2;\tau_2, \bs{Q^s})\| \,\eu^{\alpha \theta}\,\,\, \mid \theta \ge 0 ,\;  \tau_2 \ge 0,\;\bs{Q^s} \in W^s_{\lloc}(\tau_2)\} \le c_0(\delta).$$
\end{theorem}

Using the flow of~\eqref{sist-din}, with a standard argument, we pass from the local manifolds $W^u_{\lloc}(\tau_1)$
and $W^s_{\lloc}(\tau_2)$, defined for $\tau_1 \le 0 \le \tau_2$, to the global manifolds $W^u(\tau)$ and $W^s(\tau)$ defined
for any $\tau \in \R$.
In fact, rephrasing~\cite[Appendix]{FrSf2},
we set
\begin{equation}\label{defW1}
\begin{split}
     W^u(\tau):= & \displaystyle{\bigcup_{T\le 0}} \big\{ \bs{\phi}(\tau; T, \Q)  \mid \Q \in W^u_{\lloc}(T)   \big\},  \\
\noalign{\smallskip}
       W^s(\tau):= & \displaystyle{\bigcup_{T\ge 0} \big\{ \bs{\phi}(\tau; T, \Q)  \mid \Q \in W^s_{\lloc}(T)   \big\}}.
\end{split}
\end{equation}
We observe that $W^u(\tau)$ and $W^s(\tau)$, the unstable and stable leaves respectively, are $C^1$ immersed manifolds which can be characterized as follows
\begin{equation}
\label{wuu}
\begin{split}
   W^u(\tau):= & \Big\{ \QQ \mid \lim_{t\to-\infty}\, \bs{\phi}(t;\tau,\QQ)=(0,0), \; x(t;\tau,\QQ) \ge 0 \textrm{ when $t\ll 0$} \Big\}\,,\\
   W^s(\tau):= & \Big\{ \QQ \mid \lim_{t\to+\infty}\, \bs{\phi}(t;\tau,\QQ)=(0,0), \; x(t;\tau,\QQ) \ge 0 \textrm{ when $t\gg 0$} \Big\}\,.
   \end{split}
\end{equation}
By construction, if $\Q \in W^u(\tau)$, then $\bs{\phi}(t;\tau,\QQ) \in W^u(t)$ for any $t \in \R$.


From Theorem~\ref{local.mani}
 and the smooth dependence of the flow of~\eqref{sist-din} on initial data,
we get the  smoothness property of the unstable and stable manifold.
 \begin{remark}
\label{r.cuno}
Assume   $\Hinstab$, then $W^u(\tau)$ depends continuously on $\tau$. Namely,
 let $\linea$ be  a segment which intersects  $W^u(\tau_0)$  transversely in a point $\QQ(\tau_0)$, then
there is a neighborhood  $I$  of $\tau_0$ such that $W^u(\tau)$
  intersects $\linea$ in a point $\QQ(\tau)$ for any $\tau \in I$, and $\QQ(\cdot)$ is continuous. Actually, it has the same
  regularity as \eqref{sist-din}, so it is $C^1$ if \eqref{sist-din} is $C^1$ in $t$.
  Analogously if $\Hstab$ holds, then $W^s(\tau)$ depends continuously on $\tau$, and smoothly if \eqref{sist-din} is smooth in $t$.
\end{remark}

 Let us denote by $W^u(-\infty)$  the unstable manifold $\bs{\Gamma_{-\infty}}$, defined in~\eqref{gammatau}.\\
According to~\cite{BFdP} and~\cite[\S 2.2]{FrSf2}, the smoothness property of $W^u(\tau)$ observed in Remark~\ref{r.cuno} can be extended to $\tau_0=-\infty$. 

\begin{remark}
\label{r.zerotoo}
Assume $\Hinstab$,
 then
$W^u(\tau)$ depends smoothly on $\tau\in    \R\cup\{-\infty\}$.\\
In particular, let $\linea$ be a segment transversal to $\bs{\Gamma_{-\infty}}$ and let  $\bs{Q}_{\linea}(-\infty)$ be the intersection point between $\linea$ and $\bs{\Gamma_{-\infty}}$;
follow $W^u(\tau)$ from the origin towards $x>0$,
then it intersects $\linea$ transversely in a point, say $\QQ_{\linea}(\tau)$, for any $\tau \leq -N$,
 for a suitable sufficiently large
$N=N(\linea)>0$; furthermore, the function $\QQ_{\linea}(\cdot)$ is $C^1$  and
$\lim_{\tau \to -\infty} \bs{Q}_{\linea}(\tau)= \bs{Q}_{\linea}(-\infty)$.\\
 We stress that $\QQ_{\linea}(\tau)$ is uniquely defined as the first
 intersection between $W^u(\tau)$ and $\linea$, although it might not be the unique intersection,
 especially if $\linea$ is too large.
\end{remark}
Let $\linea$ be a segment transversal to $\bs{\Gamma_{-\infty}}$, take $\tau \le -N$ and
 denote by $\tilde{W}^u_{\linea}(\tau)$  the compact and connected branch of $W^u(\tau)$ between the origin and
$\QQ_{\linea}(\tau)$.
Analogously, denote by $\tilde{W}^u_{\linea}(-\infty)$ the compact connected branch of $\bs{\Gamma_{-\infty}}=W^u(-\infty)$
between the origin and $\QQ_{\linea}(-\infty)$.

\begin{lemma}\label{l-stim}
  Assume  $\Hinstab$. Let $\linea$ and $N$ be as in Remark~\ref{r.zerotoo}, then  there is $c=c(\linea)$ such that
  \begin{equation}\label{est.esp}
    \sup \{ \|\bs{\phi}(\theta+ \tau; \tau, \QQ) \|\eu^{-\alpha \theta} \mid
    \QQ \in \tilde{W}^u_{\linea}(\tau), \; \theta \le 0, \; \tau \le - N \} \le c(\linea) .
  \end{equation}
\end{lemma}
\begin{proof}
Let $\delta>0$ be the fixed constant defined in Theorem~\ref{local.mani}.
If $\linea$ is close enough to the origin so that  $\|\bs{Q}_{\linea}(-\infty)\|<\delta$, then
    by construction
$\tilde{W}^u_{\linea}(\tau) \subset W^u_{\lloc}(\tau)$
for every $\tau\leq -N$, see~\eqref{defW1}, and~\eqref{est.esp} follows straightforwardly from Theorem~\ref{local.mani1}.

Now, let $\linea$ be a generic segment transversal to $\bs{\Gamma_{-\infty}}$ satisfying
$\|\bs{Q}_{\linea}(-\infty)\|\geq \delta$, and let  $M=M(\delta,\linea)>0$ be such that
   $$\| \bs{\phi_{-\infty}}(  -M ; 0, \bs{Q}_{\linea}(-\infty))\| =
   \delta/2\,. 
   $$
Fix a point $\bs{\hat{Q}}$ belonging to $\tilde{W}^u_{\linea}(-\infty)\setminus B(\bs{0},\frac{\delta}{2})$, then
there is
$t_{\bs{\hat{Q}} }\in[-M, 0]$   such that
$\bs{\phi_{-\infty}}(  t_{\bs{\hat{Q}}} ; 0,  \bs{Q}_{\linea}(-\infty))=\bs{\hat{Q}}$. By construction,
$\|\bs{\phi_{-\infty}}( \tau -M ; \tau, \bs{\hat{Q}})\|\leq{\frac{\delta}{2}}$ for any $\tau<0,$
since
$\bs{\phi_{-\infty}}( \tau -M ; \tau, \bs{\hat{Q}} )=\bs{\phi_{-\infty}}(  -M ; 0, \bs{\hat{Q}} )=
 \bs{\phi_{-\infty}}(  -M+t_{\bs{\hat{Q}} } ; 0,  \bs{Q}_{\linea}(-\infty))$.

 Using Remark~\ref{r.zerotoo} combined with the compactness of  $\tilde{W}^u_{\linea}(\tau)\setminus W^u_{\lloc}(\tau)$, continuous dependence on initial data and parameters of the flow of~\eqref{sist-din},
 and possibly choosing a larger $N$,
 we can assume that $\bs{R}=\bs{\phi}(\tau-M;\tau, \QQ)$
   is such that
  \begin{equation}\label{precisazione}
   \|\bs{R}\| < \delta, \quad \textrm{ so that}   \quad \bs{R}\in W^u_{\lloc}(\tau-M)
  \end{equation}
   whenever $\tau <-N$ and $\QQ \in \tilde{W}^u_{\linea}(\tau)\setminus W^u_{\lloc}(\tau)$;   notice that $M$ does not depend on $\QQ$ and $\tau$, but depends on $\linea$.
  The existence of $\bs{R}$ as in~\eqref{precisazione} is trivial for any $\QQ\in W^u_{\lloc}(\tau)$, so
\eqref{precisazione}  holds for any    $\QQ \in  \tilde{W}^u_{\linea}(\tau)$.

    Thus,  Theorem~\ref{local.mani1} ensures the existence of $c_0=c_0(\delta)>0$ such that
   $$   \sup \{ \|\bs{\phi}(s+ \tau-M; \tau, \bs{Q})\| \eu^{-\alpha s}  \mid \QQ \in \tilde{W}^u_{\linea}(\tau), \; s \le 0, \; \tau \le - N   \} \le c_0.$$
  Hence, $ \|\bs{\phi}(\theta+ \tau ; \tau, \bs{Q}) \|\eu^{-\alpha \theta} \le c_0 \eu^{\alpha M}$, for any $\theta \le -M$ and $\QQ \in \tilde{W}^u_{\linea}(\tau)$,
which proves
~\eqref{est.esp} when $\theta \le -M$.

To prove~\eqref{est.esp} for $-M \le \theta \le0$,
 it is enough to recall that $\tilde{W}^u_{\linea}(\tau)$ is compact, $M>0$ is fixed and the flow of~\eqref{sist-din} is continuous;
 then the lemma follows by choosing some
   $c(\linea) \ge c_0 \eu^{\alpha M}$.
\end{proof}

The following lemma better describes the behaviour of the solutions $\bs{\phi}(\cdot; \tau, \QQ)$ departing
from a point $\QQ \in W^u(\tau)$, as $\tau \to-\infty$. Roughly speaking, we can say that such trajectories mime the autonomous dynamical system~\eqref{sist-frozen} frozen at $\tau=-\infty$.

\begin{lemma}\label{l.vicine}
Assume  $\Hinstab$. Let $\linea$ and $\QQ_\linea(-\infty)$ be as in  Remark  $\ref{r.zerotoo}$.
Then,
\begin{equation}
\label{vicine}
\lim_{\tau \to -\infty} \, \sup_{\theta\leq 0}
\left\| \bs{\phi}\big(\theta+\tau;\tau,\QQ_\linea(\tau)\big)- \bs{\phi}_{-\infty}\big(\theta;0,\QQ_\linea(-\infty)\big) \right\| = 0\,.
\end{equation}
\end{lemma}

\begin{proof}
For brevity, we set
$$
\bs{\xi_\tau}(\theta)= \bs{\phi}(\theta+\tau; \tau, \QQ_\linea(\tau))
\quad\text{and}\quad
\bs{\xi^*}(\theta)= \bs{\phi_{-\infty}}(\theta; 0,  \QQ_\linea(-\infty))\,.
$$
We argue by contradiction, and assume  that there exist $\ee>0$ and two sequences $(\theta_n)_n\,,\, (\tau_n)_n \subset {}]-\infty, 0]$, such that $\tau_n \to -\infty$ satisfying
\begin{equation}\label{ip1}
\|\bs{\xi_{\tau_n}}(\theta_n)-\bs{\xi^*}(\theta_n)\|>\ee\,.
\end{equation}
Combining Lemma~\ref{l-stim} with~\eqref{nonfiniscemai},
we easily find that
$$ \|\bs{\xi_{\tau_n}}(\theta_n)-\bs{\xi^*}(\theta_n)\| \le \|\bs{\xi_{\tau_n}}(\theta_n)\|+ \|\bs{\xi^*}(\theta_n)\| \le 2 c(\linea) \eu^{\alpha \theta_n},
\quad \forall \tau_n\leq -N .$$
 If $(\theta_n)_n$ is unbounded,
then we easily get a contradiction with~\eqref{ip1}.

 Thus, we can assume that there is $M>0$ such that
 $(\theta_n)_n\subset [-M,0]$.
 Using continuous dependence on initial data and parameters, we see that for any $\ee>0$ there exists
  $\delta(\ee)>0$ such that if $\| \QQ_\linea(\tau)-\QQ_\linea(-\infty)\|<\delta$, then
  $\| \bs{\xi_\tau}(\theta)-\bs{\xi^*}(\theta)\| <\ee$ for any $\theta \in [-M,0]$ whenever $\tau<-N$.

  Afterwards, from Remark~\ref{r.zerotoo},
 we find $N_1=N_1(\delta)>N>0$ large enough so that
  $\| \QQ_\linea(\tau)-\QQ_\linea(-\infty)\|<\delta$ for any $\tau <-N_1$, so that eventually we get
  $\| \bs{\xi_\tau}(\theta)-\bs{\xi^*}(\theta)\| <\ee$ for any $\theta \in [-M,0]$ whenever $\tau<-N_1$;
  but this is in contradiction with~\eqref{ip1}, since we are assuming that $(\theta_n)_n \subset [-M,0]$ and
  $\tau_n \to -\infty$. The Lemma is thus proved.
\end{proof}

\noindent We denote by $$\bs{\bs{\phi}}(t;d)=(x(t;d),y(t;d))$$ the trajectory of~\eqref{sist-din}
corresponding to the regular solution  $u(r;d)$ of~\eqref{eq.pla}.

According to~\cite{FrancaCAMQ, FrancaFE},  all the regular solutions correspond to trajectories in the unstable leaf.

\begin{remark}
\label{r.corresp}
Assume that $\K \in C^1$ and  $\Hinstab$ holds.
Then,
$$u(r;d) \mbox{ is a regular solution}\quad\Longleftrightarrow\quad \bs{\phi}(\tau;d) \in W^u(\tau)  \mbox{ for every } \tau\in\R.$$
Further,  fixed $\tau_0 \in \R$,
the function $\QQQ_{\bs{\tau_0}}:[0,+\infty[{}\to W^u(\tau_0)$ defined by
\linebreak
$\QQQ_{\bs{\tau_0}}(d):=\bs{\phi}(\tau_0;d)$
 is a  continuous  (bijective) parametrization of   $W^u(\tau_0)$. In particular,
$\QQQ_{\bs{\tau_0}}(0)=(0,0)$.
 Therefore,
for any $\QQ \in W^u(\tau_0)$ there is a unique $d(\QQ) \ge 0$ such that
 $\bs{\phi}(t;d(\QQ)) \equiv \bs{\phi}(t; \tau_0, \QQ)$ for any $t \in \R$.

In fact, the dependence of  $\QQQ_{\bs{\tau_0}}$ with respect to the parameter $\tau_0$ is $C^1$.

 \end{remark}
The existence of a bijective  parametrization of   $W^u(\tau)$ can be   obtained extending to the $p$-Laplacian setting the argument developed in the proof of~\cite[Lemma 2.10]{DF}, written in the case $p=2$. The smoothness of $\QQQ_{\bs{\tau_0}}(d)=\bs{\phi}(\tau_0;d)$ with respect to $\tau_0$ follows from the smoothness
of the flow of~\eqref{sist-din}.

\smallbreak

According to~\cite[Lemma 3.7]{Ni}, we can easily prove the monotonicity properties of regular solutions.  In particular,
\begin{remark}
\label{r.decri}
Any regular solution $u(r ; d)$  of~\eqref{eq.pla} is decreasing until its first zero.
\end{remark}

Fix $\linea$ as in Remark~\ref{r.zerotoo}, so that $\QQ_{\linea}(\tau)$ is well defined for any $\tau<-N(\linea)$.
From Remark~\ref{r.corresp}, we can define the function
$d_\linea:{}]-\infty, -N(\linea)[{} \to {}]0,+\infty[{}$ by setting $d_\linea(\tau): = d(\QQ_{\linea}(\tau))$ for any $\tau<-N(\linea)$. In particular,
 $\bs{\bs{\phi}}(t;d_\linea(\tau))\equiv \bs{\phi}(t; \tau , \QQ_{\linea}(\tau))$.

Now, we need the following weak version of the implicit function theorem, cf.~\cite[Theorem 15.1]{Deim}.
\begin{theorem}\label{diniDebole}
  Let $A(x,y): \R^2 \to \R$ be a function continuous along with its partial derivative $\frac{\partial A}{\partial x}(x,y)$.
  Let $A(x_0,y_0)=0 $ and $\frac{\partial A}{\partial x}(x_0,y_0) \ne 0$, then we can find $\delta>0$ and exactly one  continuous function
$x(y):[y_0-\delta, y_0+\delta] \to \R$ such that $x(y_0)=x_0$ and $A(x(y),y)=0$ for any $y \in[y_0-\delta, y_0+\delta]$.
\end{theorem}

Our next aim consists in showing the invertibility and monotonicity of $d_{\linea}$.
\begin{lemma}\label{l.corresp.new}
Assume  $\Hinstab$.
  Let $\linea$, $N=N(\linea)$ be as in Remark $\ref{r.zerotoo}$. Then, there is $D=D(\linea,N)$ such that the function
  $d_\linea:{}]-\infty, -N[{} \to {}]D, +\infty[{}$, defined by the property
 $\bs{\bs{\phi}}(t;d_\linea(\tau))\equiv \bs{\phi}(t; \tau , \QQ_{\linea}(\tau))$, is continuous, bijective, monotone decreasing,  and its inverse
 $\tau_\linea:{}]D, +\infty[{} \to {}]-\infty, -N[{}$
 is  continuous.

Furthermore,
  $\displaystyle{\lim_{\tau \to -\infty} d_\linea(\tau)}=+\infty$, and so $\displaystyle{\lim_{d\to +\infty} \tau_\linea(d)}=-\infty$.
\end{lemma}

\begin{proof}
 As noticed in Remark~\ref{r.zerotoo}, the function $d_\linea: {}]-\infty, - N[{} \to {}]0,+\infty[{}$ is well defined, due to the transversality of the first crossing between $W^u(\tau)$ and $\linea$,
for $\tau<-N$.
We   show now that $d_\linea$ admits a  continuous  inverse, by constructing it via
 Theorem~\ref{diniDebole}.

Denote by $\bar{\linea}$ the straight line containing the segment $\linea$, and let $\mathcal{D}(\QQ)$ be the smooth function which evaluates the directed distance from $\bar{\linea}$ to $\QQ$, so that $\mathcal{D}(\QQ)=0$ if and only if $\QQ \in \bar{\linea}$.
 Recalling the parametrization $\QQQ\bs{_{\tau}}$ of $W^u(\tau)$ introduced in Remark~\ref{r.corresp}, we define $A(\tau,d):= \mathcal{D}(\QQQ\bs{_{\tau}}(d))$.
Let us consider a couple $(\tau_0,d_0)$ with $\tau_0< -N$ and $d_0=d_\linea(\tau_0)$. In particular,
$\QQQ\bs{_{\tau_0}}(d_0)=\QQ_{\linea}(\tau_0)$ and so
$A(\tau_0,d_0)=\mathcal{D}( \QQ_{\linea}(\tau_0))=0$.
From the smoothness property of $W^u(\tau)$ given in Remark~\ref{r.cuno}, we can compute
\begin{align}\label{derivoA}
\frac{\partial A}{\partial \tau} (\tau_0,d_0)
& = \left\langle \nabla \mathcal{D}(\QQQ\bs{_{\tau_0}}(d_0))
\,,\,
\frac{\partial }{\partial \tau}\QQQ\bs{_{\tau_0}}(d_0)
\right\rangle=\\
& = \left\langle \nabla \mathcal{D}(\QQQ\bs{_{\tau_0}}(d_0))
\,,\,
\bs{\dot{\phi}}(\tau_0; \tau_0, \QQQ\bs{_{\tau_0}}(d_0)) \right\rangle \ne 0,
\end{align}
since $\nabla \mathcal{D}(\QQ)$ is orthogonal to $\bar\linea$ and $\bs{\dot{\phi}}(\tau_0; \tau_0, \QQQ\bs{_{\tau_0}}(d_0))$ is transversal to $\linea$.

From the previous formula we also find  that $\frac{\partial A}{\partial \tau}$ is continuous, so
 we can apply Theorem~\ref{diniDebole}, thus finding locally  a continuous function $\tau_\linea(d)$
 such that $A(\tau_\linea(d),d) \equiv 0$;  so by construction $\tau_\linea(d)$ is the local inverse of $d_\linea$.
Since the previous argument can be performed for every couple $(\tau_0,d_0)$  satisfying $d_0=d_\linea(\tau_0)$ with
$\tau_0\in{}]-\infty,-N[{}$, we can conclude that the image of $d_\linea$ is an open interval~$\mathcal U$, and that $\tau_\linea(d)$ is a global inverse.

\medbreak

We now  show  that $\lim_{\tau\to-\infty} d_\linea(\tau)=+\infty$.

Let $c>0$ be such that $\linea\subset \{(x,y) \mid x>c\}$. Then, for every $\tau<-N$
 the point $\QQ_{\linea}(\tau)=(Q_x(\tau),Q_y(\tau))$ is such that $Q_x(\tau)>c$.
Let us consider the solution $u(r;d_\linea(\tau))$ corresponding to the trajectory
 $\bs{\bs{\phi}}(t;d_\linea(\tau))\equiv \bs{\phi}(t; \tau , \QQ_{\linea}(\tau))$.
According to Lemma~\ref{l.vicine}, $u(r;d_\linea(\tau))$ is positive for $r\in [0,\eu^{\tau}]$, and from Remark~\ref{r.decri}, we deduce that
\begin{equation}\label{dtau}
  d_\linea(\tau) = u(0;d_\linea(\tau)) \ge u(\eu^{\tau};d_\linea(\tau))= Q_x(\tau) \,\eu^{-\alpha \tau} > c\,\eu^{-\alpha \tau} \to +\infty
\end{equation}
as $\tau \to -\infty$.
Finally, we find that $d_\linea$ is decreasing
and there exists  $D=D(\linea,N)$ such that $d_\linea({}]-\infty,-N[{})={}]D,+\infty[{}$.
\end{proof}

\section{Basic properties of the function $\bs{R(d)}$.}\label{basicRd}

Using the Pohozaev identity, see e.g.\cite{FrancaAM, FrancaCAMQ, KNY,NiSe},
 it is possible to obtain the following classical result.



\begin{proposition}\label{p.classical.Poho}
  Assume $\Hinc$, then all the regular solutions $u(r;d)$ of~\eqref{eq.pla} have a zero at $r=R(d)$. Hence, all the corresponding trajectories
  $\bs{\phi}(t;d)$ of~\eqref{sist-din} are such that $y(t;d)<0<x(t;d)$ when $t<T(d):=\ln R(d)$ and $x(T(d);d)=0>y(T(d);d)$.
\end{proposition}

For the proof, we refer to~\cite{DN,KN} for the Laplacian operator and  to~\cite{GHMY,KNY} for $p$-Laplacian extensions. Then, following~\cite[Theorem 4.2]{FrancaCAMQ}, we can prove the continuity of the function $R(d)$.

\begin{proposition}\label{p.Rcontinua}
  Assume $\Hinstab$,  then the set $J$ introduced in \eqref{defJ} is open and the function
  $R: J \to \R$ is continuous.
\end{proposition}

We need two further asymptotic results, which can already be found in literature in slightly different contexts.

\begin{proposition}\label{p.da0}
  Assume $\Hinstab$, and $\inf J=0$,
 then $\lim_{d \to 0^+}R(d)= +\infty$.
\end{proposition}

We refer to~\cite[Proposition 2.4]{KabYY} for the proof.
With some effort, we can generalize the previous proposition a bit.

\begin{proposition}\label{p.daD}
  Assume $\Hinstab$ and the existence of $\tilde D>0$, with $\tilde D\notin J$ which is an accumulation point for $J$. Then,
$$\lim_{\scriptsize \begin{array}{l}d\! \to \!{\tilde D}\\d\in J\end{array}}\!\!R(d)= +\infty.$$
\end{proposition}

\begin{proof}
 According to Remark~\ref{r.decri},  $u(r,\tilde{D})$ is positive and decreasing for any $r \ge 0$, since $\tilde{D}\notin J$.
Let  $\delta>0$ be as in Theorem~\ref{local.mani}, and choose  $\tau_0 \ll 0$ so that
 $$\tilde{W}^u(\tau_0):= \{ \QQQ_{\bs{\tau_0}}(d) \mid d \in [0, \tilde{D}+1]\} \subset  W^u_{\lloc}(\tau_0) \subset B(\bs{0},\delta),$$
 where $\QQQ_{\bs{\tau_0}}(d)$ is as in Remark~\ref{r.corresp}.
By construction, $R(d)>\eu^{\tau_0}$ for any $0<d \le \tilde{D}+1$.

\emph{Assume by contradiction that there are $d_n \in J$, $d_n \to \tilde{D}$, and $M \in \R$ such that
$R(d_n) \le \eu^M$.}
Let us denote by  $\ep= \frac{1}{2}\min \{x(t;\tilde{D}) \mid t \in [\tau_0, M] \}$.

 From the continuity of the parametrization $\QQQ_{\bs{\tau_0}}(\cdot)$, for any $\sigma>0$ we can find
$\bar{n}=\bar{n}(\sigma)>0$ so that
  $\|\QQQ_{\bs{\tau_0}}(d_n)-\QQQ_{\bs{\tau_0}}(\tilde{D})\| < \sigma$ when $n > \bar{n}$.
  Then, using the continuity of the flow of~\eqref{sist-din}, we can choose $\sigma=\sigma(\ep)>0$  small enough (and $\bar{n}>0$ large enough) so that
 $$ \|\bs{\phi}(t;d_n) - \bs{\phi}(t;\tilde{D})\| =\|\bs{\phi}(t; \tau_0, \QQQ_{\bs{\tau_0}}(d_n)) - \bs{\phi}(t;\tau_0, \QQQ_{\bs{\tau_0}}(\tilde{D}))\| < \ep  $$
  for any $t \in [\tau_0,M]$,
  whenever $n >\bar{n}$.
  Hence, we get
$$x(t;d_n)\ge x(t;\tilde{D}) - |x(t;d_n) - x(t;\tilde{D})| > 2\ep -\ep=\ep , $$
for any $t \in [\tau_0,M]$ and $n>\bar{n}$. Hence
$x(t;d_n)>0$ for any $t \le M$ and, consequently,
$R(d_n)>\eu^M$  for any $n >\bar{n}$: a contradiction.
The Proposition is thus proved.
\end{proof}
Let us observe also that
 $u(r;\tilde{D})$ is a Ground State and converges to $0$ as $r \to +\infty$.

\section{The asymptotic estimate of $\bs{R(d)}$ for $\bs{d}$ large.}
Now we proceed to study the behavior of $R(d)$ for $d$ large. We emphasize that the asymptotic estimates of the first zero for large initial data are mainly  based on
the crucial assumption  $\Hell$.

Our first step is to locate $\QQ_{\linea}(\tau)$ through
the function $\HH$, see Proposition~\ref{p.H0} and Remark~\ref{r.La}. Then, we use
 a Gr\"{o}nwall's argument to get a  lower and an upper bound
of the time $T(\tau, \linea)$ taken by $\bs{\phi}(t; \tau, \QQ_{\linea}(\tau))$ to cross the $y$ negative semi-axis,
see Lemmas~\ref{l.est.time} and~\ref{l.est.T1}.
 Finally, we prove the
asymptotic estimates of $R(d)$ in Proposition~\ref{p.sub}
 (for the $\ell < \ell^*_{p}$ case)   and in Propositions~\ref{p.parziale} and~\ref{p.parzialebis}
 (for the $\ell \ge \ell^*_{p}$ case).

We emphasize that assumption $\Hell$ implies that the function $\K$ is strictly increasing in a neighborhood of $r=0$, so  we can use
 the following  {\em truncation argument}.

\begin{remark}
\label{r.k}
From $\Hell$ we see that
 there exists $\hat{T}_0<0$  such that $K(\hat{T}_0)<A+1$ and
\begin{equation}
\label{tzero}
0<\frac{1}{2}B\ell\eu^{\ell t}<\dot{K}(t)<2B\ell\eu^{\ell t} \mbox{ for every }  t \le \hat{T}_0.
\end{equation}
So, we can find a function $\hat{K}:\R \to \R$ of class $C^1$ satisfying $\hat{K}=K$ in the interval ${}]-\infty,\hat{T}_0]$,
$\frac{d\!{\hat K}}{dt}(t)>0$ for every $t\in\R$, and
$$
A+1=\hat K(+\infty) := \lim_{t\to+\infty} \hat{K}(t) \in\R\,.
$$
In particular,  $\hat{K}$ satisfies $\Hell$, $\Hinc$ and $\Hstab$. Such a truncation argument will permit us to assume implicitly the validity of the previous hypotheses, when we will look for properties of system~\eqref{sist-din} in a neighborhood of $t=-\infty$ in the presence of the only hypothesis $\Hell$.
\end{remark}

\medbreak
Now, we provide some estimates of the energy $\HH(\QQ_\linea(\tau);\tau)$, where $\QQ_\linea(\tau)$ is as  in Remark~\ref{r.zerotoo}.

\begin{proposition}\label{p.H0}
 Assume $\Hell$.
Let $\linea$ be a small enough segment, transversal to $\bs{\Gamma_{-\infty}}$, and consider
 the point $\QQ_\linea(\tau)$, with $\tau<-N(\linea)$.
Then, there is a constant $\tilde{c}(\linea)>0$ such that
\begin{equation}\label{e.H0}
  \HH(\QQ_\linea(\tau);\tau)=  \tilde{c}(\linea) \, \eu^{\ell \tau} +o(\eu^{\ell \tau}) \qquad \textrm{ as $\tau \to -\infty$}.
\end{equation}
\end{proposition}

\begin{proof}
We write for brevity
$\bs{\xi_\tau}(t)=(x_\tau(t),y_\tau(t))= \bs{\phi}(t+\tau; \tau, \QQ_\linea(\tau))$ and $\bs{\xi^*}(t)=(x^*(t),y^*(t))= \bs{\phi_{-\infty}}(t; 0,  \QQ_\linea(-\infty))$.
Notice that, according to Lem\-ma~\ref{l.vicine}, $x_\tau$ is positive in $]-\infty,0]$.

 Since $\QQ_\linea(\tau) \in W^u(\tau)$, from Lemma~\ref{l.basic} combined with Remark~\ref{r.corresp} we know that $\displaystyle{\lito \HH(\bs{\xi_\tau}(t);t+\tau)=0}$.

\noindent Hence, in the spirit of~\cite[pag. 357]{FrancaFE}, by~\eqref{Hder} and Remark~\ref{r.k} we find
\begin{equation}\label{I0}
      \begin{split}
        \HH(\QQ_\linea(\tau);\tau)=  &
        \int_{-\infty}^0 \dot K(t+\tau) \frac{x_\tau(t)^q}{q} \, dt =\\
          = & \int_{-\infty}^0 \left[B\ell \eu^{\ell(t+\tau)} +
          h'(\eu^{t+\tau})\eu^{t+\tau}
          \right] \frac{x_\tau(t)^q}{q} \, dt =
          \eu^{\ell \tau} \left( I_1+I_2  \right)\,,
      \end{split}
    \end{equation}
    where
    $$ I_1:= \frac{B\ell}{q} \int_{-\infty}^0  \eu^{\ell t} x_\tau(t)^q \, dt \quad\mbox{ and } \quad  I_2:=
\frac1q \int_{-\infty}^0 \frac{
h'(\eu^{t+\tau})\eu^{t+\tau}}{\eu^{\ell (t+\tau)}}\, \eu^{\ell t} x_\tau(t)^q \, dt.$$
We can rewrite $I_1$ in the following equivalent form:
   $$   I_1 =
      \frac{B\ell}{q}  \int_{-\infty}^0 \eu^{\ell t} x^*(t)^q\, dt
      + \frac{B\ell}{q}  \int_{-\infty}^0 \eu^{\ell t} [x_\tau(t)^q - x^*(t)^q]\, dt .$$
     Recalling~\eqref{vicine} and the fact that both $x_\tau(t)$, $x^*(t)$ converge to $0$ exponentially as $t \to -\infty$ as observed in Lemma~\ref{l-stim} and in~\eqref{homosapiens}, we apply the Lebesgue theorem
to deduce the existence of $\omega_1(t)$ such that
\begin{equation}\label{I1}
I_1=\tilde c(\linea)+ \omega_1(\tau)\hspace{3mm} \mbox{ with } \hspace{3mm} \lim_{\tau\to-\infty}\omega_1(\tau) =0\,,
\end{equation}
where
$\tilde c(\linea): = \frac{B\ell}{q}  \int_{-\infty}^0 \eu^{\ell t} x^*(t)^q\, dt >0
$
is finite, due to~\eqref{homosapiens}.

 Analogously, from $\Hell$,
we find $\omega_2(t)$ such that
 \begin{equation}
\label{I2}
I_2=\omega_2(\tau)\hspace{3mm} \mbox{ with } \hspace{3mm} \lim_{\tau\to-\infty}\omega_2(\tau) =0\,.
    \end{equation}
The thesis follows by plugging~\eqref{I1} and~\eqref{I2} in~\eqref{I0}.
\end{proof}

For a fixed $a>0$, let $\linea(a)$ denote the segment  of  $y=-a$ lying between the negative $y$ semi-axis and the isocline $\dot{x}=0$ of system~\eqref{sist-din}, i.e.
\begin{equation}\label{defLa}
  \linea(a)= \{ (s, -a) \mid 0 \le s \le  \tilde{X}(a) \} ,\quad \textrm{ where $\tilde{X}(a):=\tfrac{1}{\alpha}\, {a^{\frac{1}{p-1}}}$}\,.
\end{equation}
Recalling the definition of equilibrium point $\EE(-\infty)$  given in~\eqref{EE}, we observe that for any $0<a<|E_y(-\infty)|$ the homoclinic orbit $\bs{\Gamma_{-\infty}}$ intersects $\linea(a)$ transversely.
From Lemma~\ref{l.vicine} and Proposition~\ref{p.H0}, we get the following asymptotic result.

 \begin{remark}\label{r.La}
 Assume $\Hell$.
Follow $W^u(\tau)$ from the origin towards $x>0$. Then, for any $0<a<|E_y(-\infty)|$
 we can find $N(a)>0$ such that
 $W^u(\tau)$ intersects $\linea(a)$ transversely  in a point, say $\QQ(\tau,a):= \QQ_{\linea(a)}(\tau)=(Q_x(\tau,a),-a)$, whenever $\tau<-N(a)$.
 From~\eqref{defLa} we see that
  $$\dot{y}(\tau;\tau, \QQ(\tau,a))>0>\dot{x}(\tau;\tau, \QQ(\tau,a)), \quad \mbox{for any } \tau \le -N(a).$$
 Moreover, there is a constant $c(a)>0$ such that
 \begin{equation}\label{e.Ha}
  \HH(\QQ(\tau,a);\tau)= c(a) \, \eu^{\ell \tau} +o(\eu^{\ell \tau}) \qquad \textrm{ as $\tau \to -\infty$}\,.
 \end{equation}
 \end{remark}

In the next part of this section
we study of the asymptotic behavior of the solutions,
under the more restrictive  assumptions $\Hinc$, $\Hell$ and $\Hstab$. Finally, $\Hinc$ and $\Hstab$ will be removed
using the truncation argument suggested by Remark~\ref{r.k}.

\begin{lemma}\label{l.dots}
  Assume $\Hinc$. Let  $u(r;d)$ be a regular solution of~\eqref{eq.pla} and let
  $\bs{\phi}(t;d)$ be the corresponding trajectory   of~\eqref{sist-din}. Then, there are
  $T_x(d)<T(d)$ such that $x(t;d)>0$ when $t<T(d)$ and it becomes null at $t=T(d)$; furthermore, $\dot{x}(t;d)>0$ when $t< T_x(d)$, and $\dot{x}(t;d)<0$ when $T_x(d)<t\leq T(d)$.
\end{lemma}
 \begin{proof}
 The existence of $T(d)$ follows from Proposition~\ref{p.classical.Poho}. From Lemma~\ref{l.basic}, we know that   $\HH( \bs{\phi}(t;d); t) \ge 0$ for any $t \leq T(d)$.
By a simple calculation, $\HH(\QQ;t) <0$ for every $\QQ=(Q_x,Q_y)$ in the isocline $\dot{x}=0$ with $0<Q_x\leq E_x(t)$, so we easily deduce the existence of
$T_x(d)<T(d)$ such that $\dot{x}(T_x(d);d)=0$ and $x(T_x(d);d)>E_x(T_x(d))$.
 \end{proof}

Assume $\Hell$ and $\Hinc$.
 Fixed $0<a<|E_y(-\infty)|$,  consider
 $\linea(a)$ and $N(a)>0$ as in Remark~\ref{r.La} so that
$\QQ(\tau,a)\in W^u(\tau)\cap \linea(a)$ is well defined for any  $\tau<-N(a)$.

From Proposition~\ref{p.classical.Poho}, there exists $T(\tau,a)>\tau$ such that $\bs{\phi}(t; \tau, \QQ(\tau,a))$ crosses the $y$ negative semi-axis at $t=T(\tau,a)$ in a point, say (see Figure~\ref{TT1})
\begin{equation}\label{pointR}
 \bs{R}(\tau,a)=(0, R_y(\tau,a))=\bs{\phi}(T(\tau,a); \tau, \QQ(\tau,a))\,.
 \end{equation}
According to Remark~\ref{r.k}, it is not restrictive to assume that $\dot{K}(t)>0$ for any $t \le \tau \le -N(a)$
(provided that we choose $-N(a)<\hat{T}_0$ in Remark~\ref{r.La}).
So, from Lemma~\ref{l.basic} we deduce that $\HH( \bs{Q}(\tau,a);\tau)>0$.

\begin{figure}[t]
\centering
\includegraphics[scale=0.6]{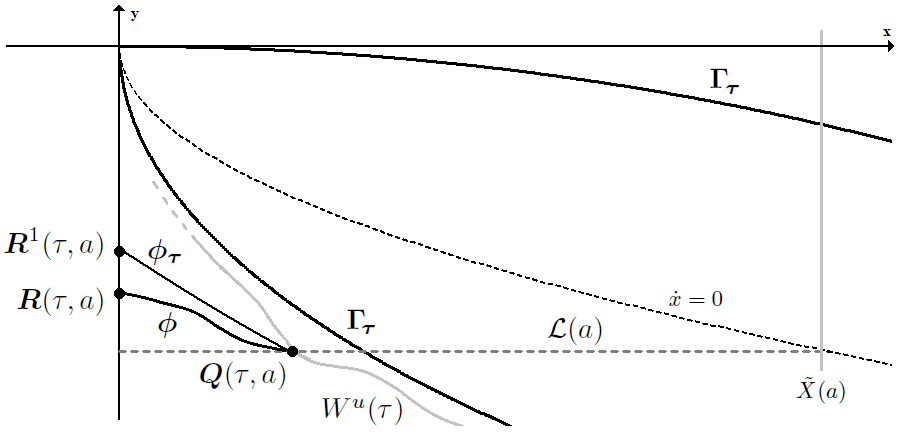}
  \caption{The position of the points $\bs{Q}(\tau,a)$, $\bs{R}(\tau,a)$, and  $\bs{R}^1(\tau,a)$.}
\label{TT1}
\end{figure}

Our next purpose is to provide suitable estimates from above and from below of $T(\tau,a)$.
From Lemma~\ref{l.dots},
 $\bs{\phi}(t; \tau, \QQ(\tau,a))$ is a graph on the $x$-axis when $t \in [\tau, T(\tau,a)]$.
In particular,
  we can find a function $\psi: [0,Q_x(\tau,a)] \to {}]-\infty,0[\,$ such that the image of $\bs{\phi}(\cdot;\tau, \QQ(\tau,a)): [\tau, T(\tau,a)] \to \R^2$ can be parametrized as $(x, \psi(x))$ for $x \in [0,Q_x(\tau,a)]$.

Let us now consider the trajectory $\bs{\phi_{\tau}}(t; \tau, \QQ(\tau,a))$ of the  system~\eqref{sist-frozen} frozen at $t=\tau$.
 Notice that its graph is contained in the level set
\begin{equation}
\label{validor}
  \{(x,y) \mid \HH(x,y; \tau)= \HH( \bs{Q}(\tau,a);\tau)>0\},
\end{equation}
 which   lies in the exterior of the homoclinic orbit $\bs{\Gamma_\tau}$ defined in~\eqref{gammatau}.

Hence, there exists $T^1(\tau,a)>\tau$
 such that $\bs{\phi_{\tau}}(\cdot ; \tau, \QQ(\tau,a))$ lies in the $4^{th}$ quadrant when $t \in [\tau, T^1(\tau,a)[{}$ and it crosses transversely the $y$ negative semi-axis at $t=T^1(\tau,a)$
  in a point, say (see Figure~\ref{TT1})
\begin{equation}\label{pointR1}
 \bs{R}^1(\tau,a)=(0, R^1_y(\tau,a))=\bs{\phi_\tau}(T^1(\tau,a); \tau, \QQ(\tau,a))\,.
\end{equation}

\begin{remark}\label{Tta-continuity}
Assume $\Hinstab$, then the functions $T(\tau,a)$ and $T^1(\tau,a)$ are continuous in their domain since  the flow on the negative $y$-axis is transversal.
\end{remark}

From the study of the trajectories of the autonomous system~\eqref{sist-frozen},
we can deduce that
\begin{equation}
\label{unain+}
\dot{x}_\tau(t;\tau,\QQ(\tau,a))<0<\dot{y}_\tau(t;\tau,\QQ(\tau,a)) \quad \mbox{for any } t\in[\tau,T^1(\tau,a)],
\end{equation}
and, consequently,
we can find a strictly decreasing function $\psi_\tau: [0,Q_x(\tau,a)] \to [-a, R^1_y(\tau,a)]$ such that the image of $\bs{\phi_\tau}(\cdot;\tau, \QQ(\tau,a)): [\tau, T^1(\tau,a)] \to \R^2$ can be parametrized as $(x, \psi_\tau(x))$ for $x \in [0,Q_x(\tau,a)]$.

\begin{lemma}\label{l.Tuno}
Assume $\Hell$ and that $\dot{K}(t)>0$ for any $t \in \R$.
Fix $0<a<|E_y(-\infty)|$,
and let $N(a)>0$ be as in Remark~\ref{r.La}. Then,
$$T(\tau,a)<T^1(\tau,a)\, , \qquad \mbox{for any } \tau<-N(a).$$
\end{lemma}

 \begin{proof}
For brevity, let us write
$\bs{\phi}(t)=(x(t),y(t))=\bs{\phi}(t;\tau, \QQ(\tau,a))$ and
$\bs{\phi_\tau}(t)=(x_\tau(t),y_\tau(t))=\bs{\phi_\tau}(t;\tau, \QQ(\tau,a))$.

From~\eqref{ripristi} combined with  Lemma~\ref{l.dots},
we see that $\HH(\bs{\phi}(\cdot);\tau)$ is strictly increasing in the interval $[\tau, T(\tau,a)]$. In particular,
recalling~\eqref{validor}, we have
$$
\HH(\bs{\phi}(t); \tau)> \HH(\bs{Q}(\tau,a); \tau)=\HH(\bs{\phi_\tau}(s);\tau)
$$
for every $t\in{}]\tau,T(\tau,a)]$ and $s\in[\tau,T^1(\tau,a)]$.
Thus,
we get
\begin{equation}\label{psisotto}
\psi(x) < \psi_{\tau}(x) \qquad \mbox{ for any } 0 \le x  < Q_x(\tau,a)\,,
\end{equation}
and, as an immediate consequence,
$$
R_y(\tau,a) < R^1_y(\tau,a)\,.
$$

 The  statement of the lemma holds if we prove that $x(t) < x_{\tau}(t)$ for any $t \in {}]\tau,T(\tau,a)]$.
 Observe first that $\bs{\phi}(\tau)=\bs{\phi_{\tau}}(\tau)$. Then, from~\eqref{sist-din} and~\eqref{sist-frozen}
 we find that $x(\tau)$ equals $x_{\tau}(\tau)$ along with its first and second derivatives, however
 $$ \frac{d^3}{dt^3}[x(\tau)-x_{\tau}(\tau)]= \frac{\ddot{y}(\tau)-\ddot{y}_{\tau}(\tau)}{p-1} |y(\tau)|^{\frac{2-p}{p-1} }=-\dot{K}(\tau) \frac{x(\tau)^{q-1}}{p-1} |y(\tau)|^{\frac{2-p}{p-1}}  <0\,. $$

 Hence, $x(t) < x_{\tau}(t)$ when $t$ is in a suitable  right neighborhood of $\tau$.
 Let
 $$\hat{T}:= \sup\{ t > \tau \mid x(s) < x_{\tau}(s)\,,\, \forall s\in{}]\tau,t]\, \}\,.$$
If $\hat{T} > T(\tau,a)$,
 the lemma is proved. So, we assume by contradiction that $\tau< \hat{T}\leq T(\tau,a)$. Then, we have $\bar x = x(\hat T)=x_\tau(\hat T)$,
   $\bs{\phi}(\hat T)=(\bar x,\psi(\bar x))$ and $\bs{\phi_{\tau}}(\hat T)=(\bar x,\psi_{\tau}(\bar x))$.
   Hence, recalling~\eqref{psisotto},
   \begin{align*}
\dot x(\hat{T})
&
= \alpha x(\hat{T})- |y(\hat{T})|^{\frac{1}{p-1}} = \alpha \bar x -|\psi(\bar x)|^{\frac{1}{p-1}}\\
&< \alpha \bar x -|\psi_\tau(\bar x)|^{\frac{1}{p-1}}= \alpha x_{\tau}(\hat{T})-|y_{\tau}(\hat{T})|^{\frac{1}{p-1}}
=\dot x_{\tau}(\hat{T}) .
\end{align*}
So   $x > x_{\tau}$ in a left neighborhood of $\hat{T}$,
 giving a contradiction.
 \end{proof}

In the following Lemma we assume $K$ bounded. This assumption will be removed via  Remark~\ref{r.k}.


\begin{lemma}\label{l.est.time}
  Assume $\Hell$, $\Hinc$, and  $\Hstab$.
For any $\ee>0$ we define
\begin{equation}\label{def.a}
  a=a(\ee):=\left(   \frac{\alpha^q \ee }{(1+\ee)\ov{K} \ }  \right)^{\frac{p-1}{q-p}} < |E_y(+\infty)|=
\left(   \frac{\alpha^q }{K(+\infty)}  \right)^{\frac{p-1}{q-p}}
\,.
\end{equation}
Then, there is $\mathcal T(\ee)$ such that, for any $\tau<\mathcal T(\ee)$,
\begin{align}
\label{eq.star1}
 T(\tau,a) &> \tau +  \frac{1}{\alpha} \ln\left(\frac{a}{|R_y(\tau,a)|}\right)\\[2mm]
\label{eq.star2}
T^1(\tau,a) &< \tau +
 \frac{(1 +\ee)}{\alpha} \ln\left(\frac{a}{|R_y^1(\tau,a)|}\right)\,.
\end{align}

\end{lemma}

\begin{proof}
Let $\mathcal T(\ee)=-N(a(\ee))$ be the value provided by Remark~\ref{r.La}, and fix $\tau <\mathcal T(\ee)$.
We consider again the trajectories
$
\bs{\phi}(t)=(x(t),y(t))=\bs{\phi}(t;\tau, \QQ(\tau,a))$ and
$\bs{\phi_\tau}(t)=(x_\tau(t),y_\tau(t))=\bs{\phi_\tau}(t;\tau, \QQ(\tau,a))$.

Observe that $\dot y>0$ along the horizontal segment $\linea(a)$. as well as in the interior of the bounded set enclosed by $\linea(a)$, $\bs{\Gamma_\tau}$ and the $y$ negative semi-axis. Using this fact, according to Lemma~\ref{l.dots} and
~\eqref{unain+}, we see that
\begin{align*}
\dot{x}(t)<0<\dot{y}(t) \,,
&& \dot{x}_{\tau}(s)<0<\dot{y}_{\tau}(s)\,,
\\
-a\leq y(t)< 0\,,
&& -a\leq y_\tau(s)< 0\,,
\end{align*}
for any $t \in [\tau, T(\tau,a)]$ and any $s  \in [\tau, T^1(\tau,a)]$.
Moreover, for both the trajectories we get
\begin{equation}\label{e.dotx}
\dot{x}<0 \quad\Rightarrow\quad\alpha x < |y|^{\frac{1}{p-1}}
\quad\Rightarrow\quad
-x^{q-1} > -\frac{|y|^{\frac{q-1}{p-1}}}{\alpha^{q-1}}\,.
\end{equation}
Then, since $\bs{\phi}(t)=(x(t),y(t))$ solves~\eqref{sist-din}, we find
\begin{equation}\label{e.doty}
 -\alpha y(t) - \frac{K(t)}{\alpha^{q-1}} |y(t)|^{\frac{q-1}{p-1}} <
\dot y(t) = -\alpha y(t) - K(t)x(t)^{q-1} < -\alpha y(t)\,,
\end{equation}
for any  $t \in [\tau, T(\tau,a)[{}$. Analogously, since $\bs{\phi_\tau}(t)=(x_\tau(t),y_\tau(t))$ solves~\eqref{sist-frozen},
\begin{equation}\label{e.dotytau}
 -\alpha y_{\tau}(t) - \frac{K(\tau)}{\alpha^{q-1}} |y_{\tau}(t)|^{\frac{q-1}{p-1}} <
\dot y_{\tau}(t) = -\alpha y_{\tau}(t) - K(\tau)x_\tau(t)^{q-1} < -\alpha y_{\tau}(t)\,,
\end{equation}
for any  $t \in [\tau, T^1(\tau,a)[{}$.

\medbreak

Note that~\eqref{def.a}
guarantees that
$$
\frac{\ov{K}}{\alpha^q} |y|^{\frac{q-p}{p-1}} \leq\frac{\ee}{1+\ee}\,, \qquad 
\mbox{for every } y\in[-a,0].
$$
Thus, we get
\begin{equation}\label{conto.a}
\begin{split}
-\alpha y- \frac{K(t)}{\alpha^{q-1}}  |y|^{\frac{q-1}{p-1}}
&\geq  -\alpha y \left(1 -\frac{\ov{K}}{\alpha^q} |y|^{\frac{q-p}{p-1}} \right)
\geq
-\frac{\alpha}{1+ \ee}  y\,,
\end{split}
\end{equation}
for every $t\in\R$ and $y \in [-a, 0]$.
 In particular,
from~\eqref{e.doty},~\eqref{e.dotytau}, and~\eqref{conto.a}  we find
\begin{equation}\label{conto.doty}\begin{split}
-\frac{\alpha}{1+ \ee} y(t) <   \dot{y}(t) < -\alpha y(t)\, ,& \qquad
-\frac{\alpha}{1+ \ee} y_{\tau}(s) <  \dot{y}_{\tau}(s) < -\alpha y_{\tau}(s)\,,
\end{split}
\end{equation}
for every $t \in [\tau, T(\tau,a)[{}$ and $s \in [\tau, T^1(\tau,a)[{}$.

 Consequently,  $\frac{\dot{y}(t)}{-\alpha y(t)}<1$ for any $t \in [\tau, T(\tau,a)[$; hence,
recalling the definition of $ \bs{R}$ in~\eqref{pointR},
 we obtain
 \begin{equation*}
    \begin{split}
       T(\tau,a)-\tau & =\int_{\tau}^{T(\tau,a)} dt >  \int_{\tau}^{T(\tau,a)} \frac{\dot{y}(t)}{-\alpha y(t)} dt =   \\
       &=  -\frac{1}{\alpha} \ln\left(\frac{y(T(\tau,a))}{y(\tau)} \right)  = \frac{1}{\alpha} \ln\left(\frac{a}{|R_y(\tau,a)|}\right)\,.
    \end{split}
  \end{equation*}
Analogously, from~\eqref{conto.doty} we also find $\frac{(1+\ep)\dot{y}_{\tau}(s)}{-\alpha y_{\tau}(s)}>1$ for any $s \in [\tau; T^1(\tau,a)[$; hence recalling the definition of $ \bs{R}^1$ in~\eqref{pointR1}, we get
  \begin{equation*}
    \begin{split}
       T^1(\tau,a)-\tau & =\int_{\tau}^{T^1(\tau,a)} ds< \int_{\tau}^{T^1(\tau,a)} \frac{(1+\ee)\dot{y}_{\tau}(s)}{-\alpha y_{\tau}(s)} ds=\\
         &=  -\frac{1+\ee}{\alpha} \ln\left(\frac{y_{\tau}(T^1(\tau,a))}{y_{\tau}(\tau)} \right)  = \frac{1+\ee}{\alpha} \ln\left(\frac{a}{|R_y^1(\tau,a)|}\right)\,.
    \end{split}
  \end{equation*}
\end{proof}
By a simple integration of~\eqref{conto.doty}, we obtain a crucial inequality.
\begin{remark}
\label{r.daprima}
 Assume $\Hell$, $\Hinc$, and $\Hstab$. Then, for every $\ee>0$,
there is $\mathcal T(\ee)$ such that, for any $\tau<\mathcal T(\ee)$, the trajectory
  $\bs{\phi}(t; \tau , \QQ(\tau,a))=(x(t),y(t))$ satisfies
\begin{align}
  \label{daprima}
  a \eu^{-\alpha (t-\tau)}  &< |y(t)|  <
  a\, \eu^{-\frac{\alpha}{1+ \ee} (t-\tau)}\,, \qquad \qquad \mbox{for every }t\in]\tau, T(\tau,a)[\,.
\end{align}
\end{remark}

\begin{lemma}\label{l.est.T1}
Assume  $\Hell$, $\Hinc$, and $\Hstab$.
 Let  $\ee>0$ and   $a=a(\ee)>0$  be as in  Lemma~\ref{l.est.time}. Then, there are $\mathcal{T}(\ee)<0$ and  a constant $C=C(a(\ee),\ee)$ such that
 \begin{equation}\label{T1}
    T^1(\tau,a(\ee))< C(a(\ee),\ee) + |\tau| \left( (1+\ee)\,\frac{\ell}{\ell_{p}^*}  - 1\right) \,,\hspace{4mm}
\mbox{for any }\tau< \mathcal{T}(\ee).
 \end{equation}
\end{lemma}

\begin{proof}
From Remark~\ref{r.La},~\eqref{Hq} and~\eqref{validor}, we deduce
the following estimates on the $y$-coordinate of the point $\bs{R}^1(\tau,a)$ introduced in~\eqref{pointR1}:
$$
\frac{p-1}{p} |R^1_y(\tau,a)|^{\frac{p}{p-1}} = \HH (\bs{R}^1(\tau,a);\tau) = \HH(\QQ(\tau,a);\tau) = c(a) \eu^{\ell \tau} + o\left(\eu^{\ell \tau}\right)\,,
$$
for any $\tau <\mathcal{T}(\ee)$, possibly choosing a larger $|\mathcal{T}(\ee)|$.
So, there is $c^1(a)= \left( c(a)\frac{p}{p-1}\right)^{\frac{p-1}{p}}>0$, which is independent of $\tau$, such that
\begin{equation}\label{Ryup}
|R^1_y(\tau,a)|=c^1(a)  \eu^{\frac{p-1}{p}\ell \tau} + o\left(\eu^{\frac{p-1}{p} \ell \tau}\right)\,,\qquad \mbox{for any }\tau <\mathcal{T}(\ee).
\end{equation}
Then, by~\eqref{eq.star2} in Lemma~\ref{l.est.time}, setting $ C(a,\ee)= \frac{1+\ee}{\alpha} \ln \left( \frac{2a}{c^1(a)}\right)$, we find
\begin{equation}\label{diseq.T1}
 \begin{split}
 T^1(\tau,a) &< \tau + \frac{1+\ee}{\alpha} \ln \left( \frac{2a}{c^1(a) \eu^{\frac{p-1}{p}\ell \tau} }\right)=\\
\noalign{\smallskip}
 & = \tau + C(a,\ee) +  |\tau| \ \ell \  \frac{1+\ee}{\alpha} \ \frac{p-1}{p} = \\
\noalign{\smallskip}
 & = C(a,\ee) + |\tau| \left(  \ell \ \frac{1+\ee}{\alpha}\  \frac{p-1}{p} - 1\right) \,,
 \end{split}
\end{equation}
for any $\tau< \mathcal{T}(\ee)$. Since
\begin{equation}\label{lpstar}
\ell_{p}^*= \frac{n-p}{p-1}= \frac{\alpha p}{p-1}\,,
\end{equation}
we conclude.
\end{proof}

\subsection{The case $\bs{\ell < \ell_{p}^*}$.}\label{S.sub-new}

\begin{proposition}\label{p.T1}
    Assume  $\Hell$, $\Hstab$ and that $\dot{K}(t)>0$ for any $t \in \R$.
If $\ell < \ell_{p}^*$, then there exists
$a>0$
 such that $$\lim_{\tau \to -\infty} T(\tau,a)=-\infty\,.$$
\end{proposition}
\begin{proof}
We simply need to choose $\ee>0$ in Lemma~\ref{l.est.T1} satisfying the
assumption $\ell(1+\ee)<\ell_{p}^*$.
Then, we set $a= a(\ee)$ as in~\eqref{def.a}, so that Lemma~\ref{l.est.T1} permits us to
 conclude the proof recalling that $T(\tau,a)<T_1(\tau,a)$ by Lemma~\ref{l.Tuno}.
\end{proof}

\begin{proposition}\label{p.sub}
Assume  $\Hell$, $\Hstab$ and that $\dot{K}(t)>0$ for any $t \in \R$.
If $\ell < \ell_{p}^*$, then there exists $D>0$ such that
    ${}]D, +\infty[{} \subset J$ and $\lim_{d \to +\infty} R(d)=0$.
\end{proposition}
\begin{proof}
Let $a=a(\ee)>0$ be as in Proposition~\ref{p.T1}.
Then, we can recover the constant $N(a)>0$ provided by Remark~\ref{r.La}, and the function $d_{\linea(a)}$ defined in Lemma~\ref{l.corresp.new} such that $\bs{\phi}(t;d_{\linea(a)}(\tau)):= \bs{\phi}(t; \tau, \QQ(\tau,a))$.
From Lemma~\ref{l.corresp.new}
there exists $D=D(a)$ such that
 the inverse function
  $\tau_{_{\linea(a)}}:{}]D, +\infty[{} \to {}]-\infty,-N(a)[{}$ is continuous and satisfies
  $\lim_{d\to+\infty}\tau_{_{\linea(a)}}(d)=-\infty$.
In particular, $\,]D, +\infty[\, \subset J$.

\medbreak

Then, from Proposition~\ref{p.T1} and Remark~\ref{Tta-continuity}, we find
$$\lim_{d \to +\infty} T(\tau_{_{\linea(a)}}\!(d) ,a)= \lim_{\tau \to -\infty} T(\tau,a)=-\infty. $$
Therefore, the first zero $R(d)$  of $u(r;d)$ satisfies
$$
\lim_{d\to +\infty} R(d) = \lim_{d\to +\infty}
\eu^{T(\tau_{_{\linea(a)}}(d), a)}=0\,,
$$
thus concluding the proof.
\end{proof}

Since the previous results focus their attention
 on a neighborhood of $\tau=-\infty$, recalling Remark~\ref{r.k} we can
remove the hypothesis  $\Hstab$ and the monotonicity assumption on $K$.

\begin{proposition}\label{p.T1-bis}
    Assume  $\Hell$.
If $\ell < \ell_{p}^*$, then there exists
$a>0$
 such that $$\lim_{\tau \to -\infty} T(\tau,a)=-\infty\,.$$
\end{proposition}

\noindent
Now, repeating the argument of the proof of Proposition~\ref{p.sub} combined with the truncation argument of the proof of Proposition~\ref{p.T1-bis}, we
obtain the asymptotic behavior of $R(d)$ for large values of $d$, which will allow us to prove
the part of Theorem~\ref{t.main} concerning the $\ell < \ell_{p}^*$ case.

\begin{proposition}\label{p.sub-bis}
     Let assumption  $\Hell$ hold with $\ell < \ell_{p}^*$. Then, there exists $\hat D>0$ such that
    ${}]\hat D, +\infty[{} \subset J$ and $\ds{\lim_{d \to +\infty} R(d)=0}$.
\end{proposition}

\subsection{The case $ \bs{\ell \geq \ell_{p}^*}$.}\label{S.super}

In this subsection we focus our attention on the opposite case  $\ell \geq \ell_{p}^*$.
We invite the reader to take in mind Lemma \ref{l.corresp.new}, i.e.
the intersection time $\tau_{\linea}(d) \to -\infty$ if and only if $d \to +\infty$ .

\begin{remark}\label{r.why}
Recalling the definition of $\alpha$ in~\eqref{transf1}
 we see that if $\ee<1/\alpha$ then
\begin{equation}
\label{ell0}
\ell_{p}^* < \frac{1}{1+\ee} \  \frac{q\alpha}{p-1}\,.
\end{equation}
\end{remark}

\begin{lemma}\label{l.H1leq2H0}
  Let assumptions $\Hell$, $\Hinc$, and $\Hstab$ hold with $ \ell \geq \ell^*_p$. Let us fix $\ee<1/\alpha$
 and define $a=a(\ee)$ as in \eqref{def.a}.
Assume that there exists a sequence $(\tau_n)_n$ with $\tau_n \to -\infty$, satisfying $\displaystyle{\lim_{n\to+\infty} T(\tau_n, a)} = -\infty$.
Then, if $n$ is sufficiently large,
$$
\HH(\RR(\tau_n, a);T(\tau_n, a))
\leq 2 \HH(\QQ(\tau_n, a); \tau_n)\,.
$$
\end{lemma}

\begin{proof}
Since both $\tau_n$ and $T(\tau_n, a)$ converge to $-\infty$,
according to Remark~\ref{r.k},
we can set
$\dot K(t)< 2B\ell \eu^{\ell t}\leq 2B\ell \eu^{\ell^*_p t}$ for every $t\in[\tau_n,T(\tau_n, a)]$.
Then,  using \eqref{Hder}, \eqref{e.dotx}, Remark \ref{r.daprima} and \eqref{ell0}, we get
\begin{equation*}
\begin{split}
\HH(\RR(\tau_n, a);T(\tau_n, a)) &\, -\HH(\QQ(\tau_n, a); \tau_n)\,=\\
&=\int_{\tau_n}^{T(\tau_n, a)} \,  \dot K(t) \frac{|x(t)|^q}{q}\, dt\,<\\
&< \frac{2B \ell}{q\alpha^q} \int_{\tau_n}^{T(\tau_n, a)} \eu^{\ell^*_p t} |y(t)|^{\frac{q}{p-1}} \, dt\,<\\
&< \frac{2B \ell}{q\alpha^q} a^{\frac{q}{p-1}} \eu^{\ell^*_p \tau_n}\,
 \int_{0}^{T(\tau_n, a)-\tau_n} \eu^{\left(\ell^*_p - \frac{\alpha}{1+\ee} \, \frac{q}{p-1} \right) s}  \, ds\,<\\
 &< \frac{2B \ell}{q\alpha^q} a^{\frac{q}{p-1}} \eu^{\ell^*_p \tau_n}
 \int_{0}^{+\infty} \eu^{\left(\ell^*_p - \frac{1}{1+\ee} \, \frac{q \alpha }{p-1} \right) s}  \, ds\,=\\
 &=\frac{2B\ell}{q\alpha^q} a^{\frac{q}{p-1}}
 \,\frac{1}{\frac{1}{1+\ee} \, \frac{q\alpha}{p-1}-\ell^*_p}\, \eu^{\ell^*_p \tau_n} =: C_\HH(\ee) \ \eu^{\ell^*_p \tau_n}\,, \\
\end{split}
\end{equation*}
when $n$ is sufficiently large.

We argue by contradiction, and assume that there exists a subsequence of $\tau_n$, still called $\tau_n$ for simplicity, which satisfies
$$
\HH(\RR(\tau_n, a);T(\tau_n, a)) \,  > 2 \HH(\QQ(\tau_n, a); \tau_n)\,,
$$
so that, for sufficiently large $n$,
\begin{equation*}
\begin{split}
\HH(\RR(\tau_n, a);T(\tau_n, a)) \, & < 2\left[ \HH(\RR(\tau_n, a);T(\tau_n, a))  -\HH(\QQ(\tau_n, a);\tau_n) \right]\,\leq\\
& \leq 2 C_\HH(\ee) \  \eu^{\ell^*_p \tau_n}\,.
\end{split}
\end{equation*}
Then, recalling the definition of $\HH$ in~\eqref{Hq} and the definition of  $\bs{R}$ in~\eqref{pointR}, we get
$$
|R_y(\tau_n, a)| < \hat C_\HH(\ee) \ \eu^{\ell^*_p \frac{p-1}{p} \tau_n}\,,
$$
where $\hat C_\HH(\ee)= \left[ 2 C_\HH(\ee) \, \frac{p}{p-1} \right]^{\frac{p-1}{p}}$. Hence,
setting $\widetilde{C}_\HH(\ee):= {\frac{1}{\alpha}} \ln \left( \frac{a}{\hat C_\HH(\ee)}\right)$,
from~\eqref{eq.star1} and~\eqref{lpstar}
it follows that
\begin{equation*}
\begin{split}
T(\tau_n,a) \, & > \tau_n +\frac{1}{\alpha} \ln \left( \frac{a}{\hat C_\HH(\ee)} \eu^{-\ell^*_p \frac{p-1}{p} \tau_n}\right)\,= \\ & = \widetilde{C}_\HH(\ee) + \tau_n\left( 1- \ell^*_p\, \frac{p-1}{p\alpha} \right) =  \widetilde{C}_\HH(\ee)\,,
\end{split}
\end{equation*}
which contradicts our assumption $\ds{\lim_{n\to+\infty} T(\tau_n, a) = -\infty}$.
\end{proof}

\begin{lemma}\label{l.H1leq2H0-bis}
  Let assumption $\Hell$ hold with $\ell \geq \ell_{p}^*$. Fix $\ee<1/\alpha$ and define
  \begin{equation}\label{def.a.hat}
  a=a(\ee):=\left(   \frac{\alpha^q \ee }{[K(-\infty)+1](1+\ee)}  \right)^{\frac{p-1}{q-p}} \,.
\end{equation}

Let us assume that there exist a sequence $(\tau_n)_n$ with $\tau_n \to -\infty$, such that $\displaystyle{\lim_{n\to+\infty} T(\tau_n, a)} = -\infty$.
Then, if $n$ is sufficiently large,
$$
\HH(\RR(\tau_n, a);T(\tau_n, a))
\leq 2 \HH(\QQ(\tau_n, a); \tau_n)\,.
$$
\end{lemma}

\begin{proof}
Arguing as in Remark~\ref{r.k}, we can modify system~\eqref{sist-din}, replacing $K$ with $\hat{K}$ and notice that we can apply Lemma~\ref{l.H1leq2H0} in this case.
From the hypothesis, we deduce that, for $n$ large, we have $\tau_n<T(\tau_n,a)<\hat{T}_0$. Since
$\hat K$ and the original $K$ coincide  on ${}]-\infty,\hat{T}_0]$,  we easily conclude.
\end{proof}

\begin{proposition}\label{prop.liminf}
Assume   $\Hell$ with $ \ell \geq \ell_{p}^*$.
Fix  $\ee<1/\alpha$, and define
$a=a(\ee)$ as in \eqref{def.a.hat}.
Then,
$$
\liminf_{\tau\to-\infty} T(\tau,a) > - \infty\,,
$$
where we set $T(\tau,a):=+\infty$, if the corresponding trajectory $\bs{\phi}(t; \tau, \QQ(\tau,a))$ does not cross the negative $y$-axis.
\end{proposition}

\begin{proof}
We argue by contradiction assuming
the existence of a sequence $(\tau_n)_n$, with $\tau_n \to -\infty$ and satisfying $\lim_{n\to+\infty} T(\tau_n,a) = -\infty$.

Since  we are focusing our attention on a neighborhood of $t=-\infty$, we can again suitably modify $K$ in the function $\hat K$ as suggested in Remark \ref{r.k}, in order to ensure the validity of the hypotheses of Lemma~\ref{l.H1leq2H0}.

Hence, we can assume,
without loss of generality, that $\Hinc$ and $\Hstab$ hold, too.
So, according to~\eqref{Hder}, the energy $\HH$ is increasing along the trajectories.
As a consequence, from Lemma~\ref{l.H1leq2H0-bis} and Remark~\ref{r.La}, we get
\begin{align*}
\tfrac12 c(a) \eu^{\ell \tau_n} \leq
\HH(\QQ(\tau_n,a);\tau_n) & \leq
\HH(\RR(\tau_n,a);T(\tau_n,a)) \\ & \leq
2 \HH(\QQ(\tau_n,a);\tau_n) \leq
3 c(a) \eu^{\ell \tau_n},
\end{align*}
for $n$ sufficiently large.
So, recalling
 the definition of $\HH$ in~\eqref{Hq}  and the definition of  $\bs{R}$ in~\eqref{pointR},
we can find a positive constant $\check c_1(a)$ such that
$$
|R_y(\tau_n,a)| \leq \check c_1(a)
\eu^{\frac{p-1}{p}\, \ell \tau_n}\,.
$$
Moreover, we can use the estimate in Remark~\ref{r.daprima} to get
$$
a \eu^{-\alpha(T(\tau_n,a)-\tau_n)} \leq |R_y(\tau_n,a)| \,,
$$
for $n$ sufficiently large. Hence,  we have
$$
 a \eu^{-\alpha(T(\tau_n,a)-\tau_n)} \leq \check c_1(a)
\eu^{\frac{p-1}{p}\, \ell \tau_n}
\,, \qquad%
$$
leading to
\begin{align*}
T(\tau_n,a) &\geq \frac{1}{\alpha} \ln \left(\frac{a}{\check c_1(a)}\right)  + \tau_n \left[ 1- \frac{1}{\alpha} \, \frac{p-1}{p}\, \ell \right]\\
& = \check C_1(\ee)  + \tau_n \left[ 1- \frac{\ell}{\ell_{p}^*}  \right]  \geq \check C_1(\ee)\,.
\end{align*}
which is in contradiction with
$\lim_{n\to+\infty} T(\tau_n,a) = -\infty$.
The proposition is thus proved.
\end{proof}


When $\ell >\ell_{p}^*$ (with the strict inequality), we find a more precise estimate.

\begin{proposition}\label{prop.lim}
Assume $\Hell$ with $\ell >\ell_{p}^*$.
Fix $\ee>0$ with $\ee<1/\alpha$ and define $a=a(\ee)$ as in~\eqref{def.a.hat}.

Assume that there is $\widetilde N(a)>0$ such that, for every $\tau<-\widetilde N(a)$, there exists a time $T(\tau,a)\in\R$ at which the trajectory
$\bs{\phi}(t; \tau, \QQ(\tau, a))$  crosses the $y$ negative semi-axis, and
$\dot{x}(t; \tau, \QQ(\tau, a))<0$
for any $\tau \le t \le T(\tau, a)$.
Then
$$
\lim_{\tau\to-\infty} T(\tau, a) = + \infty\,.
$$
\end{proposition}
\begin{proof}
Let $\check T :=\liminf_{\tau\to-\infty} T(\tau,a)\in\R\cup\{+\infty\}$, given by Proposition~\ref{prop.liminf}.
Without loss of generality, we choose
the value $\hat{T}_0 $ provided by Remark~\ref{r.k} so that $\hat{T}_0\leq \check T$.
Recalling Remark~\ref{r.why}, we first fix $\delta\in{}]0,1[{}$ satisfying
\begin{equation}\label{pre-ell0bis}
\ell_{p}^* < \frac{1-\delta}{1+\ee} \  \frac{q\alpha}{p-1}\,,
\end{equation}
and then $\ov\tau_{\delta}<-\widetilde N(a)$ such that
$$
\tau < \delta \tau < \hat{T}_0 -1\leq \check T-1< T(\tau,a) , \quad \mbox{for any }\tau<\ov\tau_{\delta}.
$$
Since $\ell  >  \ell_{p}^*$, we can introduce $\ell_1 \leq \ell$
  such that
\begin{equation}\label{ell0bis}
\ell_{p}^* < \ell_1 \leq \frac{1-\delta}{1+\ee} \  \frac{q\alpha}{p-1}  < \frac{1}{1+\ee} \  \frac{q\alpha}{p-1}\,.
\end{equation}

Consider the trajectory of system~\eqref{sist-din} departing from $\QQ(\tau, a)$ at the time $\tau$ and   the points
$$
\SSS(\tau, a) = {\bs \phi}(\delta\tau;\tau,\QQ(\tau, a))\,,
\qquad
\RR(\tau, a) = {\bs \phi}(T(\tau, a);\tau,\QQ(\tau, a))\,.
$$
When we focus our attention on the trajectory ${\bs \phi}(\cdot;\tau,\QQ(\tau, a))$ restricted to the interval $[\tau,\delta\tau]\subset{}]-\infty,\hat{T}_0[{}$, recalling the truncation argument in Remark~\ref{r.k}, we can assume that both $\Hinc$ and $\Hstab$ hold.
So, we can argue as in the proof of Lemma~\ref{l.H1leq2H0},
  to obtain  $\mathcal T(\ee,\delta)\leq \ov\tau_{\delta}$ such that
\begin{align}
\nonumber
\HH(\SSS(\tau, a);\delta \tau) &\, -\HH(\QQ(\tau, a); \tau)=
\\
\nonumber
&=\int_{\tau}^{\delta\tau} \,  \dot K(t) \frac{|x(t)|^q}{q}\, dt\leq
\\
\nonumber
&\leq \frac{2B \ell}{q\alpha^q} \int_{\tau}^{\delta \tau} \eu^{\ell_1 t} |y(t)|^{\frac{q}{p-1}} \, dt\leq
\\
\label{step3a}
&\leq \frac{2B \ell}{q\alpha^q} a^{\frac{q}{p-1}} \eu^{\ell_1 \tau}
 \int_{0}^{+\infty} \eu^{\left(\ell_1 - \frac{\alpha}{1+\ee} \, \frac{q}{p-1} \right) s}  \, ds = C_\HH(\ee) \ \eu^{\ell_1 \tau},
\end{align}
for any $\tau<\mathcal T(\ee,\delta)$.
Moreover, from~\eqref{daprima} and $\dot x(\delta\tau)<0$, we have
\begin{equation}\label{est.taudelta-a}
x(\delta\tau) < \frac{1}{\alpha} |y(\delta\tau)|^{\frac{1}{p-1}}
<\frac{1}{\alpha} \,a^{\frac{1}{p-1}}\, \eu^{\frac{1-\delta}{1+\ee}\frac{\alpha}{p-1}\tau}
\quad \mbox{for any }\tau<\mathcal T(\ee,\delta).
\end{equation}
Assume by contradiction that there is $M>0$ and a sequence $(\tau_n)_n$ such that $\tau_n\to-\infty$ and $T(\tau_n,a)\leq M$.

Let us now focus our attention on  $[\delta\tau_n, T(\tau_n,a)]\subset{}]-\infty,M]$.
We remark that the validity of  $\Hinc$ and $\Hstab$ is not guaranteed anymore in  $[\delta\tau_n, T(\tau_n,a)]$;
however $\dot x <0$, so $x(t) < x(\delta\tau_n)$ in this interval. Hence, we can compute
\begin{multline}
\HH(\RR(\tau_n, a);T(\tau_n, a)) \, - \HH(\SSS(\tau_n, a);\delta \tau_n)
=\int_{\delta\tau_n}^{T(\tau_n, a)} \,  \dot K(t) \frac{|x(t)|^q}{q}\, dt\leq
\\
\label{step3b}
\leq \frac{1}{q} \int_{\delta\tau_n}^{T(\tau_n, a)} \max\{ \dot K(t), 0\} |x(\delta\tau_n)|^q \, dt := \frac{\Delta(\tau_n)}{q\alpha^q}\,.
\end{multline}
Then, using \eqref{est.taudelta-a}, we find
\begin{align*}
  \Delta(\tau_n)
&  \leq  \, \left[\int_{-\infty}^{M} \max\{ \dot K(t), 0\} \, dt\right] a^{\frac{q}{p-1}}\, \eu^{\frac{1-\delta}{1+\ee} {\frac{q\alpha}{p-1}}\tau_n}= \\
& := \mathcal{I}_K(M) \, a^{\frac{q}{p-1}}\, \eu^{\frac{1-\delta}{1+\ee} {\frac{q\alpha}{p-1}}\tau_n}\,.
\end{align*}
Plugging this last inequality in~\eqref{step3b} and using~\eqref{ell0bis}, we obtain
\begin{equation}\label{step3c}
\begin{split}
 \HH(\RR(\tau_n, a);T(\tau_n, a)) &\, - \HH(\SSS(\tau_n, a);\delta \tau_n) \le \mathcal{I}_K(M) a^{\frac{q}{p-1}} \eu^{\ell_1 \tau_n}\,,
\end{split}
\end{equation}
for $n$ sufficiently large.

Hence, summing~\eqref{step3a} and~\eqref{step3c} with the   estimate  in Remark~\ref{r.La}, since $\ell\geq \ell_1$, we get the existence of a constant $\tilde C_\HH(\ee)>0$ satisfying
$$
\frac{p-1}{p} |R_y(\tau_n,a)|^{\frac{p}{p-1}} = \HH(\RR(\tau_n, a),T(\tau_n, a))
\leq  \tilde C_\HH(\ee) \eu^{\ell_1 \tau_n}\,.
$$

Finally, since $\dot y< -\alpha y$ in $[\tau_n,T(\tau_n,a)[{}$, and so $-\frac{1}{\alpha}\, \frac{\dot y}{y} <1$,  for a certain constant $\widetilde{C}_H(\ee)$, we obtain
 \begin{equation*}
    \begin{split}
       T(\tau_n,a)& > \tau_n -\frac{1}{\alpha} \int_{\tau_n}^{T(\tau_n,a)} \frac{\dot y(s)}{y(s)}\,ds\,=\\
       &= \tau_n -\frac{1}{\alpha}\int_{a}^{|R_y(\tau_n,a)|} \frac{dy}{y}  = \tau_n + \frac{1}{\alpha} \ln\left(\frac{a}{|R_y(\tau_n,a)|}\right)\geq\\
       &\geq \widetilde{C}_H(\ee) + \tau_n\left( 1- \frac{\ell_1}{\ell_{p}^*} \right) \,,
    \end{split}
  \end{equation*}
when $n$ is large.
Recalling~\eqref{ell0bis}, we get $\lim_{n\to+\infty} T(\tau_n,a)= +\infty$, giving a contradiction. The Proposition is thus proved.
\end{proof}

\begin{proposition}\label{p.parziale}
Assume $\Hell$ and $\Hinc$.
Then, $J={}]0,+\infty[{}$ and
\begin{itemize}
  \item  if $\ell= \ell_{p}^*$, then $\displaystyle{\liminf_{d \to +\infty} R(d)>0}$,
  \item  if $\ell> \ell_{p}^*$, then $\displaystyle{\lim_{d \to +\infty} R(d)=+\infty}$.
\end{itemize}
\end{proposition}

\begin{proof}
As stated in Proposition \ref{p.classical.Poho},
$J={}]0,+\infty[{}$ is a well-known consequence of assumption $\Hinc$.

Let $\ep < \frac{1}{\alpha}$, define $a=a(\ee)$ as in~\eqref{def.a.hat}, and set $N=N(a)$ as in Remark~\ref{r.La}. So, $W^u(\tau)$ crosses transversely $\linea(a)$ in $\QQ(\tau, a)$
 for any
$\tau<-N$.

From Proposition~\ref{p.classical.Poho} and Lemma~\ref{l.dots},
we see that for every $\tau<-N$ there is $T(\tau, a)$ such that
$\dot{x}(t; \tau , \QQ(\tau, a))<0$ for any $\tau \le t \le T(\tau, a)$,
and $\bs{\phi}(t; \tau , \QQ(\tau, a))$ crosses the $y$ negative semi-axis at $t=T(\tau,a)$.

Now, recalling Lemma~\ref{l.corresp.new}, we consider the function $d_{\linea(a)}:
{}]-\infty,- N[{} \to {}]D, +\infty[{}$ such that $u(r;d_{\linea(a)}(\tau))$ is the solution of~\eqref{eq.pla}
corresponding to
\linebreak
$\bs{\phi}(t; \tau , \QQ(\tau, a))$ via~\eqref{transf1}, and its inverse $\tau_{\linea(a)}$ satisfying $\lim_{d \to +\infty} \tau_{\linea(a)}(d)= -\infty$.

If $\ell \geq\ell_{p}^*$, Proposition~\ref{prop.liminf} gives $\liminf_{\tau \to -\infty}T(\tau,a)>-\infty$. Arguing as in the proof of Proposition~\ref{p.sub},
we have
\begin{equation}\label{lim1}
\liminf_{d\to+\infty} R(d) = \liminf_{d\to+\infty} \eu^{T(\tau_{\linea(a)}(d),a)}= \liminf_{\tau\to -\infty} \eu^{T(\tau, a)} >0\,.
\end{equation}

If $\ell >\ell_{p}^*$, we are able to apply Proposition~\ref{prop.lim} and get
\begin{equation}\label{lim2}
\lim_{d\to+\infty} R(d) = \lim_{d\to+\infty} \eu^{T(\tau_{\linea(a)}(d),a)}= \lim_{\tau\to -\infty} \eu^{T(\tau, a)} =+\infty\,.
\end{equation}
The proposition is thus proved.
\end{proof}
Theorem \ref{t.main.due} follows from Proposition \ref{p.parziale}, cf. \S \ref{proof} for more details.

In order to prove the second part of Theorem \ref{t.main}, we   need to remove
the monotonicity assumption from Proposition \ref{p.parziale}.
Note that if $\Hinc$ does not hold, we cannot ensure
 the existence of points of $J$ in a neighborhood of $+\infty$,
when $\ell \geq  \ell_{p}^*$.
 However, we can prove a weaker result.

\begin{proposition}\label{p.nuova}
Assume $\Hell$, and define
$$
\tilde R(d) = \begin{cases}
R(d) & \mbox{if } d\in J\,,\\
+\infty & \mbox{if } d\notin J\,.
\end{cases}
$$
If $\ell \geq \ell_{p}^*$, then there is $\RDIR_0>0$ such that  $\tilde R(d) \ge \RDIR_0$ for any $d\geq 0$.
\end{proposition}
\begin{proof}
  Fix $0<\ee<1/\alpha$, $a=a(\ee)$ as in~\eqref{def.a.hat} and set $N(a)$ as in Remark \ref{r.La}.
  Let $\hat{T}_0$ be as in Remark \ref{r.k}.
From Proposition \ref{prop.liminf} we deduce that
\begin{equation}\label{aaa}
\mbox{there is $T' \le \hat{T}_0$ such that $T(\tau,a) \ge T'$ for any } \tau < -N(a)\,.
\end{equation}




  Now, let us fix $\tau_0<-N(a)$ and  denote by $\tilde{W}^u(\tau_0)$ the branch of the unstable manifold $W^u(\tau_0)$ between the origin and
 $\Q(\tau_0,a)$.  From Remark \ref{r.corresp} we see that there is $D^*>0$ such that the function
  $\QQQ_{\bs{\tau_0}}(d)$ restricted to $d \in[0, D^*]$ gives a parametrization of $\tilde{W}^u(\tau_0)$, i.e.  $\QQQ_{\bs{\tau_0}}:[0,D^*]\to \tilde{W}^u(\tau_0)$
  is a continuous bijective function.

 Using Remark \ref{r.zerotoo} and Lemma \ref{l.vicine}, reducing $\tau_0$ if necessary, we see that if $\Q \in \tilde{W}^u(\tau_0)$ then
 $\bs{\phi}(t;\tau_0 , \Q)$ is close to the corresponding trajectory in $W^u(-\infty)$ when $t \le \tau_0$; in particular
 $x(t;\tau_0 , \Q)>0$ if $t\le \tau_0$. Hence $\tilde R(d) \ge \eu^{\tau_0}$ for any $d \le D^*$.

Let now $d>D^*$. Recalling Remark \ref{r.corresp}, we have $\QQQ_{\bs{\tau_0}}(d) \in
  W^u(\tau_0) \setminus \tilde{W}^u(\tau_0)$, and using
Lemma \ref{l.corresp.new}  we see that  there is a unique
 $\tau_{\linea(a)}(d)<\tau_0$ such that the trajectory
$\bs{\phi}(t;\tau_0, \QQQ_{\bs{\tau_0}}(d))$ crosses transversely $\linea(a)$ at $t= \tau_{\linea(a)}(d)<\tau_0< -N(a)$.
  Then, from \eqref{aaa} we see that
  $$T(\tau_{\linea(a)}(d),a)  \ge T' \qquad \mbox{for any } d > D^* \,.$$
  Hence $\tilde R(d) \ge \eu^{T'}$ for any $d > D^*$.

  Then, setting $\RDIR_0 = \min \{\eu^{\tau_0} , \eu^{T'} \}$, the Proposition is proved.
\end{proof}

\bigbreak

Now, we reprove Proposition~\ref{p.parziale} by replacing the global assumption $\Hinc$ with the local assumption $\ell \le \frac{n}{p-1}$.
Motivated by~\cite{BE1, CL97,LL},
and inspired by~\cite{FrancaTMNA} and~\cite{FrancaFE}, we obtain the following result.

\begin{lemma}\label{l.new.technical}
  Assume $\Hell$ with $\ell \le \frac{n}{p-1}$; assume further either $\Hstab$ or that
 there is $\rho_1>0$ such that $\K(r) \ge \K(\rho_1)$ when $r \ge \rho_1$.
Then, there is $\hat D>0$ such that ${}]\hat D,+\infty[{} \subset J$.
\end{lemma}

Notice that if $\Hell$ with $\ell \le \frac{n}{p-1}$ holds and $\lim_{r \to +\infty} \K(r)$ exists and it is positive, either bounded or
unbounded, or if $K(t)$ is asymptotically periodic, then Lemma~\ref{l.new.technical} applies.

The proof of Lemma~\ref{l.new.technical}  is rather technical and it is postponed to Section~\ref{S.new.technical}, as well as the related proof of the following adapted version of Proposition \ref{p.parziale}.

\begin{proposition}\label{p.parzialebis}
  Assume $\Hell$ with $\ell_{p}^*<\ell \le \frac{n}{p-1}$; assume further either $\Hstab$ or that
 there is $\rho_1>0$ such that $\K(r) \ge \K(\rho_1)$ when $r \ge \rho_1$.
Then $$\displaystyle{\lim_{d \to +\infty} R(d)=+\infty}\,.$$
\end{proposition}

\subsubsection{Proof of Lemma~\ref{l.new.technical} and Proposition~\ref{p.parzialebis}.}\label{S.new.technical}

We start by proving Lemma~\ref{l.new.technical} under assumption $\Hstab$:
the following arguments are preliminary to the proof under this hypothesis.
    The alternative case  where it is assumed that $\K(r) \ge \K(\rho_1)$ is then obtained as a Corollary.

\medbreak

Let us assume $\Hstab$, so that we can construct the stable manifold $W^s(\tau)$, see \S\ref{basic} and, in particular,~\eqref{wuu}.

To develop our construction we need to define  several sets, and we invite the reader to follow the argument on Figure~\ref{fig.tecnico}.
Assume $\Hell$ and $\Hstab$,  then
  for any $\tau \in \R$ we set
\begin{align*}
    \ov{\HH}(x,y)&= \alpha xy + \frac{p-1}{p} |y|^{\frac{p}{p-1}} + \ov{K}\,\frac{|x|^q}{q}\,,
    &&
    \\
    \und{\HH}(x,y)&= \alpha xy + \frac{p-1}{p} |y|^{\frac{p}{p-1}} + \und{K}\,\frac{|x|^q}{q}\,.
    &&
\end{align*}
We define $\ov{\Gamma}:=\{(x,y) \mid \ov{\HH}(x,y)=0 \; x \ge 0 \}$ and
$\und{\Gamma}:=\{(x,y) \mid \und{\HH}(x,y)=0 \; x \ge 0  \}$. Notice that both $\ov{\Gamma}$ and $\und{\Gamma}$ are the image
of closed regular curves; we denote by $\ov{F}$ and by $\und{F}$ the bounded sets enclosed by
$\ov{\Gamma}$ and $\und{\Gamma}$, respectively: notice that $\ov{F}\subset\und{F}$.
We denote by $\bs{\ov{G}}= (\ov{G}_x,\ov{G}_y)$ the (transversal) intersection between  $\ov{\Gamma}$ and the isocline $\dot{x}=0$ such that $\ov{G}_x>0$.
Then, we denote by $\bs{\und{G}}$ the intersection between the line $x=\ov{G}_x$ and $\und{\Gamma}$ contained in $\dot{x}<0$ and by
$\mathcal{G}$ the vertical segment between $\bs{\und{G}}$ and $\bs{\ov{G}}$. Moreover, we denote by $\partial \mathcal{B}^{\uparrow}$ the branch of
$\ov{\Gamma}$ between the origin and  $\bs{\ov{G}}$ contained in $\dot{x} \le 0$
 and by $\partial \mathcal{B}^{\downarrow}$ the branch of
$\und{\Gamma}$ between the origin and  $\bs{\und{G}}$ contained in $\dot{x} \le 0$.
Finally, we denote by $\mathcal{B}$ the compact set enclosed by $\mathcal{G}$, $\partial \mathcal{B}^{\uparrow}$ and $\partial \mathcal{B}^{\downarrow}$, see Figure~\ref{fig.tecnico}.

We emphasize that if $\bs{\phi}(t)=(x(t),y(t))$ is a trajectory of~\eqref{sist-din}, we find,
according to \eqref{ripristi},
\begin{align*}
   \frac{d}{dt} \,\ov{\HH}(\bs{\phi}(t))
   &=  \big( \,\ov{K}-K(t) \big)  \,  x(t)\,| x(t)|^{q-2} \,\, \dot{x}(t)\,, \\
   \frac{d}{dt}\,  \und{\HH}(\bs{\phi}(t))
   &=  \big( \,\und{K}-K(t) \big)  \,  x(t)\,| x(t)|^{q-2} \,\, \dot{x}(t)\,.
\end{align*}
Using this fact, we easily obtain the following crucial remark
\begin{remark}\label{r.box}
  Assume $\Hell$ and $\Hstab$, then the flow of~\eqref{sist-din}  on $\breve{\mathcal{G}}:=\mathcal{G} \setminus \{\bs{\ov{G}}, \bs{\und{G}} \}$
  aims towards the interior of $\mathcal{B}$ for any $t \in \R$, while on $(\partial \mathcal{B}^{\uparrow} \cup \partial \mathcal{B}^{\downarrow}) \setminus \{(0,0) \}$
  aims towards the exterior of $\mathcal{B}$ for any $t \in \R$.
\end{remark}

\begin{figure}[t]
\centerline{
\epsfig{file=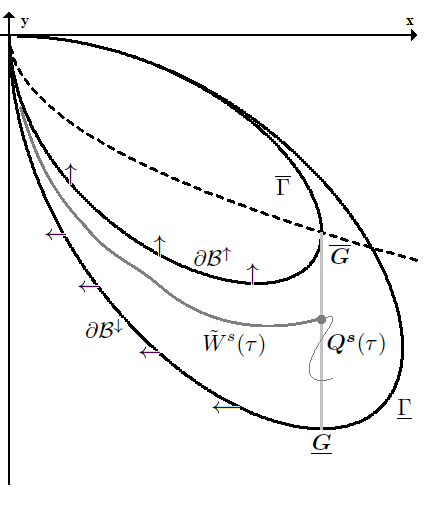, height = 6 cm}
\qquad
\epsfig{file=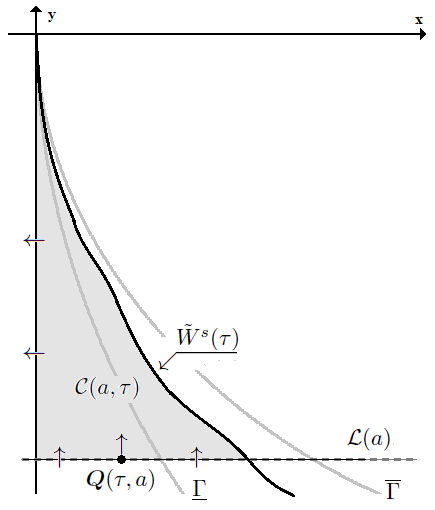, height = 6 cm}
}
\caption{The constructions needed in the proof of Lemma~\ref{l.new.technical} }
\label{fig.tecnico}
\end{figure}

Let us now fix $\tau \in \R$. From Remark~\ref{r.box} and the construction in \S\ref{basic},
 we see that $W^s(\tau)$
 intersects  $\breve{\mathcal{G}}$ (not necessarily transversely), see \cite[Lemma 2.8]{FrancaFE}.
  Follow $W^s(\tau)$ from the origin towards  $x>0$: denote by $\bs{Q^s}(\tau)$ the first intersection with $\breve{\mathcal{G}}$
  and by $\tilde{W}^s(\tau)$ the connected branch of $W^s(\tau)$ between the origin and $\bs{Q^s}(\tau)$.

Moreover, we introduce the set
  \begin{equation}\label{hatWs}
   \hat{W}^s(\tau):= \{ \QQ \in  W^s(\tau) \cap \mathcal B \mid \dot x(t; \tau, \QQ)<0 \; \textrm{ for any } t \ge \tau \} \supset \tilde W^s(\tau)\,.
  \end{equation}
In particular, using Remark~\ref{r.box}, we have
\begin{align}
    \label{What-inv}
    \QQ\in \hat{W}^s(\tau) & \Rightarrow \bs{\phi}(t;\tau,\QQ) \in \hat{W}^s(t) \mbox{ for every } t\geq \tau\,,
\\
    \QQ\in \tilde{W}^s(\tau) & \Rightarrow \bs{\phi}(t;\tau,\QQ) \in \tilde{W}^s(t) \mbox{ for every } t\geq \tau\,.
    \label{Wtilde-inv}
\end{align}

In what follows, we recall the argument in~\cite[pp. 357--360]{FrancaFE}.
Let us consider the autonomous systems~\eqref{sist-din}, where $K(t)$ is replaced by $\ov K$, respectively $\und K$, and we denote by
$$
\bs{\ov\phi}(t;\tau,\QQ) =(\ov{x}(t;\tau,\QQ),\ov{y}(t;\tau,\QQ))\,,
 \mbox{ resp. }  \bs{\und\phi}(t;\tau,\QQ) =(\und{x}(t;\tau,\QQ),\und{y}(t;\tau,\QQ))
$$
the trajectories of these systems starting at time $\tau$ from the point $\QQ$.
From \eqref{homosapiens}, recalling that in the region $\dot{x}<0$ the branch of
$\und{\Gamma}$ lies under the corresponding branch of $\ov{\Gamma}$,
see Figure \ref{fig.tecnico},
we can find $C>c>0$ such that
$$
    c\, \eu^{-\frac{n-p}{p(p-1)} t} \leq \und{x}(t;0,\bs{\und{G}}) \leq
    \ov{x}(t;0,\bs{\ov{G}}) \leq
    C\, \eu^{-\frac{n-p}{p(p-1)} t}\,, \mbox{ for any $t\geq 0$,}
$$
or, equivalently,
\begin{equation}\label{Cc-base}
    c\, \eu^{-\frac{n-p}{p(p-1)} (t-\tau)} \leq \und{x}(t;\tau,\bs{\und{G}}) \leq
    \ov{x}(t;\tau,\bs{\ov{G}}) \leq
    C\, \eu^{-\frac{n-p}{p(p-1)} (t-\tau)}\,, \mbox{ for any } t\geq \tau\,.
\end{equation}
Such estimates permit us to provide analogous ones for the solutions of the non-autonomous system~\eqref{sist-din}.

  \begin{lemma}\label{l.abovebelow}
    Assume  $\Hell$ and $\Hstab$. Let $\tau\in\R$ and $\bs{\hat Q}\in \hat W^s(\tau)\cap \mathcal{G}$, then
    \begin{equation}\label{Cc}
      c\, \eu^{-\frac{n-p}{p(p-1)}(t-\tau)} \le x(t; \tau, \bs{\hat Q}) \le C\, \eu^{-\frac{n-p}{p(p-1)}(t-\tau)}\,, \qquad \mbox{for any } t\ge \tau\,.
    \end{equation}
  \end{lemma}
\begin{proof}
We will prove that
$$
\und{x}(t;\tau,\bs{\und{G}}) \leq
x(t; \tau, \bs{\hat Q})
\leq \ov{x}(t;\tau,\bs{\ov{G}})\,, \mbox{ for any } t\geq \tau\,.
$$
Then, the conclusion will follow from~\eqref{Cc-base}.

We give the proof of the first inequality, the other being similar. Let $T=\sup\{t\geq \tau \,:\, \und{x}(s;\tau,\bs{\und{G}}) \leq
x(s; \tau, \bs{\hat Q}) \mbox{ for any } s\in[\tau,t]\,\}$. Assume by contradiction that $T\in\R$. Then, $\und{x}(T;\tau,\bs{\und{G}}) = x(T; \tau, \bs{\hat Q})$ and since $\bs{\hat Q}\in \hat W^s(\tau)$, from~\eqref{What-inv} and Remark \ref{r.box}
 the trajectory remains in the interior
  of $\mathcal B$ and in particular we get $\und{y}(T;\tau,\bs{\und{G}}) < y(T; \tau, \bs{\hat Q})$.
So, from \eqref{sist-din} and \eqref{sist-frozen}, we deduce that
$\dot{\und{x}}(T;\tau,\bs{\und{G}}) < \dot x(T; \tau, \bs{\hat Q})$, providing $\und{x}(t;\tau,\bs{\und{G}}) < x(t; \tau, \bs{\hat Q})$ in a right neighborhood of $T$, leading to a contradiction.
\end{proof}

From Lemma~\ref{l.abovebelow}, we obtain the following result, which has already been proved   in~\cite[Lemma 3.1]{FrancaFE} focusing the attention on the point $\bs{Q^s}(\tau)\in \tilde W^s(\tau)$, but we repeat the argument here to correct some typos.

  \begin{lemma}\label{l.segnoHs}
  Assume $\Hstab$ and $\Hell$ with $\ell  \le \frac{n}{p-1}$.
 Then, there is $N_1>0$ such that
  $\HH(\QQ;\tau)<0$ for any  $\QQ \in \hat{W}^s(\tau)$ and $\tau <-N_1$.
\end{lemma}

\begin{proof}
Let $\hat{T}_0$ be as in Remark \ref{r.k}, so that \eqref{tzero} holds.

Take any $\tau \leq \hat {T}_0$ and consider a point $\QQ \in \hat{W}^s(\tau)$.
We denote, for brevity, the trajectory $\bs{\phi}(t; \tau,\QQ)$ as $\bs{\phi}(t)=(x(t),y(t))$.

We note that $\bs{\phi}$ must cross the line $\mathcal G$
  at a time smaller than $\tau$.
 So, let us denote by $T_{\mathcal G} \le \tau$ the largest value with such a property. In particular, by construction, $x(T_{\mathcal G})=\ov G_x$,
 $0<x(t)<\ov G_x$ and $\dot x(t)<0$ for every $t > T_{\mathcal G}$.
Using Lemma~\ref{l.abovebelow}, we have
\begin{equation}\label{Cc-mod}
c\, \eu^{-\frac{n-p}{p(p-1)}(t-T_{\mathcal G})} \leq x(t) \leq C\, \eu^{-\frac{n-p}{p(p-1)}(t-T_{\mathcal G})}\,, \mbox{ for any } t \geq T_{\mathcal G}\,.
\end{equation}
Since $T_{\mathcal G} \le \tau \le \hat {T}_0$ we see that
 $x$ is decreasing in the interval $]\hat{T}_0,+\infty[{}\subset ]T_{\mathcal G},+\infty[{}$; so using~\eqref{Cc-mod},
\begin{align*}
 \int_{\hat{T}_0}^{+\infty} \dot K(t) \frac{x(t)^q}{q}\, dt
    &= -K(\hat{T}_0) \frac{x(\hat{T}_0)^q}{q} +
    \int_{\hat{T}_0}^{+\infty} K(t) \frac{d}{dt} \left[-\frac{x(t)^q}{q}\right]\, dt\,\geq\\
    &\geq-K(\hat{T}_0) \frac{x(\hat{T}_0)^q}{q}\, \geq\, -\frac{C^q}{q}\, K(\hat{T}_0) \, \eu^{-\frac{q(n-p)}{p(p-1)}(\hat{T}_0-T_{\mathcal G})}=\\
    &= -\frac{C^q}{q}\, K(\hat{T}_0) \, \eu^{-\frac{n}{p-1}(\hat{T}_0-T_{\mathcal G})}\,.
\end{align*}

Moreover, using~\eqref{tzero}
and~\eqref{Cc-mod},
\begin{align*}
\int_{\tau}^{\hat{T}_0} \dot K(t) \frac{x(t)^q}{q}\, dt    &\geq \frac{B\ell c^q}{2q} \int_{\tau}^{\hat{T}_0} \eu^{\ell t} \, \eu^{-\frac{q(n-p)}{p(p-1)}(t-T_{\mathcal G})}\, dt=\\
&= \frac{B\ell c^q}{2q} \eu^{\frac{n}{p-1} T_{\mathcal G}} \int_{\tau}^{\hat{T}_0} \eu^{-\left(  \frac{n}{p-1}-\ell \right)\, t}\, dt\,.
\end{align*}
Summing up, setting $b_1=\frac{C^q}{q}\,K(\hat{T}_0)
\, \eu^{-\frac{n}{p-1}\hat{T}_0}$
 and $b_2=\frac{B\ell c^q}{2q}$, from \eqref{Hder}, we get
\begin{align*}
    \HH(\QQ;\tau) &= -\int_{\tau}^{+\infty} \dot K(t) \frac{|x(t)|^q}{q}\, dt\,=\\
    &= -\int_{\hat{T}_0}^{+\infty} \dot K(t) \frac{x(t)^q}{q}\, dt
     -\int_{\tau}^{\hat{T}_0} \dot K(t) \frac{x(t)^q}{q}\, dt\,\leq\\
    &\leq
    \eu^{\frac{n}{p-1}T_{\mathcal G}} \left[ b_1 - b_2
    \int_{\tau}^{\hat{T}_0} \eu^{-\left(  \frac{n}{p-1}-\ell \right)\, t}\, dt
    \right]\,.
\end{align*}
So, since $\ell \leq \frac{n}{p-1}$, the integral diverges as $\tau\to-\infty$, thus giving the proof.
\end{proof}

Let us assume now $\Hell$ and $\Hstab$. Recalling the definition of $\EE$  given in \eqref{EE}, fix
\begin{equation}\label{fix-a}
0<a<
\left(\frac{\alpha^q}{\ov{K}}\right)^{\frac{p-1}{q-p}}
\leq
\inf\{|E_y(t)| \,:\, t\in\R\}
\end{equation}
and consider the segment $\linea(a)$ defined as in~\eqref{defLa}.

From Remark~\ref{r.La}, there is $N(a)>0$ such that $W^u(\tau)$ intersects transversely $\linea(a)$ in a point
 denoted by $\QQ(\tau,a)$, for every $\tau<-N(a)$.  Moreover,
since $a>\ov{G}_y$,  a subsegment of $\linea(a)$ joins  $\partial \mathcal{B}^{\downarrow}$ with $\partial \mathcal{B}^{\uparrow}$ (transversely), and, consequently,
from the previous argument,
$\tilde{W}^s(\tau)$ intersects $\linea(a)$, too, for every $\tau\in\R$.

\begin{lemma}\label{l.tecnico1}
  Assume $\Hstab$ and $\Hell$ with $\ell  \le \frac{n}{p-1}$.
 Then, 
  there is $\hat{N}(a)>0$ such that for any $\tau<- \hat{N}(a)$ the trajectory
  $\bs{\phi}(t; \tau, \QQ(\tau,a))$ corresponds to a crossing solution, i.e. there is $T=T(\tau,a)>\tau$ such that
  $x(t; \tau, \QQ(\tau,a))>0$ for any $t<T$ and it becomes null at $t=T$. Further
  $\dot{x}(t; \tau, \QQ(\tau,a))<0$ for any $\tau \le t \le T(\tau,a)$.
\end{lemma}
\begin{proof}

  Let us consider the compact set $\mathcal{C}=\mathcal{C}(a,\tau)$, see Figure~\ref{fig.tecnico},  delimited by $\tilde{W}^s(\tau)$,
 the $y$ negative semi-axis and the line $\linea(a)$, for any $\tau \in\R$.
  Then, from Remark~\ref{r.La}, we can find $\hat N(a)> N_1$, with $N_1$ provided by Lemma~\ref{l.segnoHs}, such that
  \begin{equation}\label{percontr}
  \HH(\QQ(\tau,a);\tau)>\frac{c(a)}{2} \eu^{\ell \tau} >0\, \qquad \mbox{for every } \tau<-\hat N(a).
  \end{equation}
Hence, from Lemma~\ref{l.segnoHs}, we get $\QQ(\tau,a) \in \mathcal{C}(a,\tau)$ for any $\tau<-\hat{N}(a)$.

Fix $\tau <-\hat{N}(a)$; then by construction $\bs{\phi}(t; \tau, \QQ(\tau,a)) \in \mathcal{C}(a,t)$, when
$t$ is in a sufficiently small right neighborhood of $\tau$, see Remark \ref{r.cuno}.

So, there is $T(\tau,a)>\tau$ such that $\bs{\phi}(t; \tau, \QQ(\tau,a)) \in \mathcal{C}(a,t)$ for any $\tau \le t \le T(\tau,a)$
and it leaves $\mathcal{C}(a,t)$ when $t>T(\tau,a)$, or  $\bs{\phi}(t; \tau, \QQ(\tau,a)) \in \mathcal{C}(a,t)$ for any $t \ge \tau$
(i.e. $T(\tau,a)=+\infty$).

From an analysis of the phase portrait we see that
\begin{equation}\label{ddd}
\dot{x}(t; \tau, \QQ(\tau,a))<0<\dot{y}(t; \tau, \QQ(\tau,a)) \, \quad \textrm{when $\tau \le t \le T(\tau,a)$}\,.
\end{equation}
Assume first that $T(\tau,a) \in \R$, then either $\bs{\phi}(t; \tau, \QQ(\tau,a))$ crosses the $y$ negative semi-axis at $t=T(\tau,a)$
proving the thesis, or $\bs{R}:=\bs{\phi}(T(\tau,a); \tau, \QQ(\tau,a)) \in \tilde{W}^s(T(\tau,a))$.

So, assume the latter and observe that
$\bs{R}\in \hat{W}^s(T(\tau,a))$, thus   $\QQ(\tau,a) \in {W}^s(\tau)$.
Then, $\dot{x}(t; T(\tau,a),\bs{R}) <0$ for any $t> T(\tau,a)$, and by construction
$\dot{x}(t; \tau, \QQ(\tau,a))<0$ for any $\tau \le t \le T(\tau,a)$. So, since the trajectories  $\bs{\phi}(\,\cdot\,; \tau, \QQ(\tau,a))$ and $\bs{\phi}(\,\cdot\,; T(\tau,a),\bs{R})$ coincide, we conclude that
$\dot{x}(t; \tau, \QQ(\tau,a))<0$ for any $t \ge \tau$, thus giving us $\QQ(\tau,a) \in \hat{W}^s(\tau)$.

Hence, we get a contradiction comparing~\eqref{percontr} with Lemma~\ref{l.segnoHs}.

\smallbreak

We consider now the remaining case: $T(\tau,a) =+\infty$. Recalling~\eqref{ddd}, the only reasonable conclusion is that $\bs{\phi}(t; \tau, \QQ(\tau,a)) \to (0,0)$ as $t \to +\infty$, i.e. $\QQ(\tau,a) \in {W}^s(\tau)$.
Whence, from~\eqref{ddd} we find
 $\QQ(\tau,a) \in \hat{W}^s(\tau)$, and get again the contradiction comparing~\eqref{percontr} with Lemma~\ref{l.segnoHs}.

So, the lemma is proved.
\end{proof}

Taking advantage of this result, we can extend it to a wider class of functions.

\begin{lemma}\label{l.tecnico2}
    Assume $\Hell$ with $\ell \le \frac{n}{p-1}$, and that
there is $\rho_1>0$ such that $\K(r) \ge \K(\rho_1)$ when $r \ge \rho_1$.
 Then, the same conclusion as in Lemma~\ref{l.tecnico1} holds.
\end{lemma}

\begin{proof}
Recalling that $K(t):=\K(\eu^t)$, by assumption  there is $\tau_1\in\mathbb{R}$
  such that $K(t) \ge K(\tau_1)$ when $t \ge \tau_1$.
Let us define
$$
K_m(t):=
\begin{cases}
K(t) & \textrm{if }  t \le \tau_1\,, \\
K(\tau_1)  & \textrm{if }    t \ge \tau_1\,.
\end{cases}
$$
 Let
$$
\bs{\phi_m}(t;\tau,\QQ(\tau,a))=\big(x_m(t;\tau,\QQ(\tau,a)),y_m(t;\tau,\QQ(\tau,a))\big)$$
be the trajectory
of the modified system~\eqref{sist-din} where $K(t)$ is replaced by $K_m(t)$.
Observe that $K_m(t)$ satisfies $\Hstab$ and we can still apply Remark \ref{r.cuno}, but $W^s(\tau)$ depends continuously and not smoothly on $\tau$.

Since the original system and the modified one coincide for $t\leq \tau_1$, then their unstable manifolds are the same in this interval.

Concerning the modified system, let
$\ov{K}_m> \sup_{t \in \R}K_m(t)= \sup_{t \le \tau_1} K(t)$ and select the point $\bs{\ov{G}_m}=(\ov{G}_{x,m},\ov{G}_{y,m})$ as above.
Then, from Remark~\ref{r.La} we see that for any $0<a< \left( \frac{\alpha^q}{\ov{K}_m} \right)^{\frac{p-1}{q-p}}$, cf.~\eqref{fix-a},  there is $N={N}(a)>|\tau_1|>0$, such that $W^u(\tau)$ crosses transversely $\linea(a)$ in $\QQ(\tau,a)$ for any $\tau<-{N}$.

From Lemma~\ref{l.tecnico1}
 applied to the modified system,
 we can find $\hat{N}(a)\geq N(a)$ such that for any $\tau<-\hat{N}(a)$ there are $T_m(\tau,a)\in\R$ and
 $\bs{R_m}=(0,R_m)$ such that
 $\bs{\phi_m}(T_m(\tau,a);\tau,\QQ(\tau,a))=\bs{R_m}$ and both
 ${x}_m(t;\tau,\QQ(\tau,a))>0$ and
$\dot{x}_m(t;\tau,\QQ(\tau,a))<0<\dot{y}_m(t;\tau,\QQ(\tau,a))$ hold for any $\tau \le t \le T_m(\tau,a)$, see \eqref{ddd}.

 Since $\bs{\phi_m}(t;\tau,\QQ(\tau,a)) \equiv \bs{\phi}(t;\tau,\QQ(\tau,a))$ for any $t \le \tau_1$, if $T_m(\tau,a)\leq \tau_1$, the thesis is achieved. So, we assume $T_m(\tau,a)>\tau_1$; since $K_m$ is constant then $\HH(\bs{\phi_m}(t;\tau,\QQ(\tau,a)); \tau_1)$ is constant   when $t\geq \tau_1$,  see \eqref{ripristi}, then
$$
   \HH(\bs{\phi_m}(t;\tau,\QQ(\tau,a)); \tau_1) \equiv   \HH(\bs{R_m};\tau_1)= \tfrac{p}{p-1}|R_m|^{p/(p-1)}>0\,,
$$
for any $t\in[\tau_1,T_m(\tau,a)]$.

  Let us set
 $\bs{S}=(S_x,S_y):= \bs{\phi}(\tau_1;\tau,\QQ(\tau,a))= \bs{\phi_m}(\tau_1;\tau,\QQ(\tau,a))$, then from the previous estimate we get $\HH(\bs{S}; \tau_1)=      \HH(\bs{R_m};\tau_1)>0$.
Notice that we can rewrite the image of $\bs{\phi_m}(t;\tau,\QQ(\tau,a))$ as a graph in $x$, i.e. there is a decreasing smooth function $f_m$ such that
 \begin{equation*}
   \begin{split}
        \mathcal{F}:= & \{\bs{\phi_m}(t;\tau,\QQ(\tau,a)) \mid  \tau_1 \le t \le T_m(\tau,a)\}
      =   \{(x,y) \mid y\!=\! f_m(x) , \; 0 \le x \le S_x \}.
   \end{split}
 \end{equation*}
Let us define
 $$T(\tau,a)= \sup \{ s\geq \tau_1 \mid \dot x(t;\tau,\QQ(\tau,a))<0< x(t;\tau,\QQ(\tau,a)) \, \textrm{ for any $\tau_1 \le t \le s$}\};$$
from~\eqref{ripristi} we have
$$
\HH(\bs{\phi}(t;\tau,\QQ(\tau,a));\tau_1) \geq
\HH(\bs{S};\tau_1)= \HH(\bs{R_m};\tau_1)>0\,,
$$
for every $\tau_1 \le t \le T(\tau,a)$.
  Therefore
 we see that the trajectory $\bs{\phi}(t;\tau,\QQ(\tau,a))$ is forced to stay in the unbounded  set
 \begin{equation*}
   \begin{split}
   \Lambda = & \{ (x,y) \mid  0\le x \le S_x , \; y <0 ,\;  \HH(x,y;\tau_1) \ge \HH(S_x,S_y;\tau_1)\} \\
   = &\{(x,y) \mid 0\le x \le S_x \, ; y \le f_m(x)\},
      \end{split}
 \end{equation*}
 whenever $\tau_1 \le t \le T(\tau,a)$.
 Observe now that the maximum of $\dot{x}(x,y)= \alpha x -|y|^{1/(p-1)}$ within $\Lambda$ is obtained in $\mathcal{F}$ which is compact; further it has to be negative, so there is
 $C>0$ such that $\dot{x}(t;\tau,\QQ(\tau,a))\le -C$ when  $\tau_1 \le t \le T(\tau,a)$.
 Then it is easy to check that $\bs{\phi}(t;\tau,\QQ(\tau,a))$ is bounded when  $\tau_1 \le t \le T(\tau,a)$.
So, from elementary considerations, we see that $\bs{\phi}(t;\tau,\QQ(\tau,a))$ crosses the $y$ negative semiaxis at $t=T(\tau,a)<\tau_1+ S_x/C$.

\end{proof}

The previous lemmas permit us to complete the proof of Lemma~\ref{l.new.technical}.

\begin{proof}[Proof of Lemma~\ref{l.new.technical}]
We fix $a$ as in~\eqref{fix-a}, and consider the segment $\linea(a)$ defined as in~\eqref{defLa}.
From Remark~\ref{r.La}, there is $N(a)>0$ such that $W^u(\tau)$ intersects $\linea(a)$ in $\QQ(\tau,a)$ for every $\tau<-N(a)$.

Using Lemma~\ref{l.corresp.new}, we consider the decreasing continuous function
$d_{\linea(a)}:{}]-\infty, -N(a)[{} \to {}]D,+\infty[{}$ such that the solution $u(r,d_{\linea(a)}(\tau))$ of~\eqref{eq.pla} corresponds to the trajectory $\bs{\phi}(t;\tau,\QQ(\tau,a))$ of~\eqref{sist-din}.

Then, either Lemma~\ref{l.tecnico1} or~\ref{l.tecnico2} applies providing the value $\hat N(a)\geq N(a)$ such that the solution $u(r,d_{\linea(a)}(\tau))$ is a crossing solution for every $\tau< -\hat N(a)$. Setting $\hat D$ such that $d_{\linea(a)}({}]-\infty,-\hat N(a)[{}) = {}]\hat D, +\infty[{}$,
we get ${}]\hat D, +\infty[{}\subset J$. The lemma is thus proved.
\end{proof}

\begin{proof}[Proof of Proposition~\ref{p.parzialebis}]
Take $\ep < 1/\alpha$ such that $a=a(\ee)$ defined in \eqref{def.a.hat} satisfies the inequality
 \eqref{fix-a}. Define the segment $\linea(a)$ as in \eqref{defLa}, and $\hat N(a)$, $T(\cdot,a)$ as in
Lemma \ref{l.tecnico1} or \ref{l.tecnico2}. Set $\hat D$ as in the proof of Lemma \ref{l.new.technical}.

Recalling  Lemmas~\ref{l.corresp.new} and~\ref{l.new.technical},
we have  $\tau_{\linea(a)} ( {}]\hat D, +\infty[{})
={}]-\infty,-\hat N(a)[{}$ and $\lim_{d \to +\infty} \tau_{\linea(a)}(d)= -\infty$. Moreover, the function $T(\cdot,a)$ is well defined in ${}]-\infty,-\hat N(a)[{}$, and it is continuous, cf. Remark~\ref{Tta-continuity}.

Hence, we are able to apply  Proposition~\ref{prop.lim} to infer~\eqref{lim2}.

\end{proof}

\section{Proof of the theorems}
\label{proof}
Since all the preliminaries are well-established, we conclude our paper by giving the explicit proof of our main theorems.

\begin{proof}[Proof of Theorem~\ref{t.main}]
If $\ell < \ell_{p}^*$, from Proposition~\ref{p.sub-bis} there is $\hat D>0$ such that ${}]\hat D, +\infty[{}\subset J$ and $\lim_{d\to +\infty} R(d)=0$.
So, there is $\tilde{D} \ge 0$ such that $\tilde{D}\not\in J$ and $]\tilde{D}, +\infty[ \subset J$; whence from Propositions~\ref{p.da0} and~\ref{p.daD}
we have $\lim_{d\to \tilde{D}^+} R(d)=+\infty$. Further,
since $R$ is continuous in $J$, see Proposition \ref{p.Rcontinua}, we find $R(J)={}]0,+\infty[{}$: i.e.
for every $\RDIR>0$ there is $d\in J$ such that $R(d)=\RDIR$, which amounts to say that $u(r;d)$ solves problem~\eqref{Problem}.

The second assertion
follows immediately from Proposition \ref{p.nuova}.
\end{proof}

\begin{proof}[Proof of Theorem~\ref{t.main.due}]
The proof takes advantage of Propositions~\ref{p.da0} and~\ref{p.parziale}.

If $\ell = \ell_{p}^*$, we distinguish two alternatives: either $R({}]0,+\infty[{})=[\RDIR_0,+\infty[{}$ where $\RDIR_0$ is an internal minimum of $R$ or
$R({}]0,+\infty[{})={}]\RDIR_0,+\infty[{}$ where $\RDIR_0=\liminf_{d\to +\infty} R(d)>0$.
Then, the proof is concluded.

On the other hand, if $\ell > \ell_{p}^*$,
 from $\lim_{d\to 0} R(d)=\lim_{d\to +\infty} R(d)=+\infty$, we deduce that the function $R$ has an internal minimum $\RDIR_0$, and
the pre-image $R^{-1}(\RDIR)$ has at least two elements for every $\RDIR>\RDIR_0$, thus giving the multiplicity result.
\end{proof}

\begin{remark}
We want to underline that in the critical case $\ell = \ell_{p}^*$ we are not able to discern which of the alternatives analyzed in the proof of Theorem~\ref{t.main.due} holds, and, consequently,
we are not able to say if there is a solution for $\RDIR=\RDIR_0$.
\end{remark}

\begin{proof}[Proof of Theorem~\ref{t.main.tre}]
From Proposition~\ref{p.nuova} and Lemma~\ref{l.new.technical}, the set $J$ contains a nontrivial interval, and there exists $\RDIR_0:=\inf_{d\in J} R(d)>0$.
Then, the proof follows the lines of the one of Theorem~\ref{t.main.due}, profiting from
Propositions~\ref{p.Rcontinua},~\ref{p.da0}, \ref{p.daD} and~\ref{p.parzialebis}.
\end{proof}

\section*{Acknowledgements}

Francesca Dalbono was partially supported by the PRIN Project 2017JPCAPN ``Qualitative and quantitative aspects of nonlinear PDEs'' and by FFR 2022-2023 from University of Palermo.

All the authors are members of INdAM-GNAMPA.

\newpage


\begin{thebibliography}{OSS2}

\bibitem{BFdP}
Bam\'{o}n  R., Flores I., del Pino M., Ground states of semilinear elliptic equations: a geometric approach,
 Ann. Inst. H. Poincar\'e Anal. Non Lin\'eaire, {\bf 17} (2000), 551-581.

 \bibitem{BJ}
  Battelli F., Johnson R., Singular ground states for the scalar curvature equation in $\R^n$,  Differential Integral Equations, {\bf 14} (2001),  141-158.

\bibitem{BE1}
Bianchi G., Egnell H.,
An ODE approach to the equation $\Delta u+K u^{\frac{n+2}{n-2}} =0$, in $\R^n$,
Math. Z.,  {\bf  210} (1992),  137-166.


\bibitem{Bv}
Bidaut-V\'eron M.F., Local and global behavior of solutions of quasilinear equations of  Emden-Fowler type,
Arch. Rational Mech. Anal., {\bf 107} (1989), 293-324.

 \bibitem{BrN}  H. Brezis, L. Nirenberg,  Positive solutions of nonlinear elliptic equations involving critical Sobolev exponents, Comm. Pure Appl. Math., \textbf{36} (4)  (1983), 437-477

 \bibitem{BN} C. Budd, J. Norbury,  Semilinear elliptic equation and supercritical growth, J. Differential Equations
 \textbf{68}, (1987) n. 2, 169-167.

\bibitem{CL97}
Chen C.C., Lin C.S.,
Estimates of the conformal scalar curvature equation via the method of moving planes,
Comm. Pure Appl. Math., {\bf 50} (1997),  971-1017.




\bibitem{CL1}
Chen C.C., Lin C.S.
Blowing up with infinite energy of conformal metrics on $\mathbb{S}^n$,
Comm. Partial Differential Equations, {\bf 24} (1999), 785-799.


\bibitem{ChCh}
 Cheng K.S., Chern J.L., Existence of positive entire solutions of some semilinear elliptic equations, J.
Differerential Equations, {\bf 98} (1992), 169–180.

\bibitem{CHS}
Chtioui H., Hajaiej H., Soula M.,
The scalar curvature problem on four-dimensional manifolds,
Commun. Pure Appl. Anal., {\bf 19} (2020),  723-746.

\bibitem{CoLe}
Coddington E., Levinson N., Theory of Ordinary Differential Equations. Mc Graw Hill, New York, 1955.



\bibitem{DF}
Dalbono F., Franca M., Nodal solutions for supercritical Laplace equations, Commun. Math. Phys., {\bf 347} (2016), 875-901.


\bibitem{Deim}
 Deimling K., Nonlinear functional analysis. Springer, Berlin, 1985.

\bibitem{DN}
Ding, W.Y., Ni, W.M., On the elliptic equation $\Delta u+K u^{\frac{n+2}{n-2}} =0$  and related topics, Duke Math. J., {\bf 52}  (1985), 485-506.


 \bibitem{DFl}  J. Dolbeault, I. Flores,  Geometry of Phase space and solutions of semilinear elliptic
equations in a ball, Trans. Am. Math. Soc. \textbf{359},  (2007) n. 9, 4073-4087 

 \bibitem{FFna} (MR3373577) [10.1016/j.na.2015.04.015]
   I. Flores and M. Franca,
 Phase plane analysis for radial solutions to supercritical quasilinear elliptic equations in a ball,
Nonlinear Anal., \textbf{125} (2015), 128-149. 


\bibitem{Fow}
 Fowler R.H., Further studies of Emden's and similar differential equations,  Q. J. Math., {\bf 2}
(1931), 259-288.

\bibitem{FrancaTMNA}
Franca M., Non-autonomous quasilinear elliptic equations and Wa\.{z}ewski's principle,
 Topol. Methods Nonlinear Anal., \textbf{23}, (2004), 213-238.

\bibitem{Fcorri}
 Franca M., Corrigendum and addendum to ``Non-autonomous quasilinear elliptic equations and Wa\.{z}ewski's principle'', Topol. Methods Nonlinear Anal., {\bf 56} (2020),  1-30.


\bibitem{FrancaAM}
  Franca  M., Classification of positive solution of $p$-Laplace
  equation with a  growth term,  Arch. Math.  (Brno), \textbf{40}, (2004), 415-434.

\bibitem{FrancaCAMQ}
Franca M.,  Fowler transformation and radial solutions for quasilinear elliptic equations. I. The subcritical and the supercritical case,
Can. Appl. Math. Q. {\bf 16} (2008),  123-159.

\bibitem{FrancaFE}
  Franca M., Structure theorems for positive radial solutions of the generalized scalar curvature
equation,  Funkcial. Ekvac., {\bf 52} (2009),  343-369.





\bibitem{FJ}
Franca M.,  Johnson R., Ground states and singular ground states for quasilinear partial differential equations with
critical exponent in the perturbative case, Adv. Nonlinear Stud.,  {\bf 4} (2004),  93-120.

\bibitem{FrSf2}
Franca M.,   Sfecci A.,  Entire solutions of superlinear problems with indefinite weights and Hardy potentials,  J. Dynam. Differential Equations, {\bf 30} (2018),  1081-1118.

 \bibitem{FLS}
 Franchi B., Lanconelli E.,  Serrin J.,  Existence and uniqueness of
 nonnegative solutions of quasilinear equations in $\R^n$, Adv. Math.,  {\bf 118}  (1996) 177-243.

\bibitem{GHMY}
 Garc\'ia-Huidobro M., Man\'asevich R, Yarur C.S.,
 On the structure of positive radial solutions to an equation containing a $p$-Laplacian with weight,
 J. Differential Equations, {\bf 223} (2006),  51-95.

\bibitem{Gaz}
Gazzola F.,
Critical exponents which relate embedding inequalities with quasilinear elliptic problems,
Discrete Contin. Dyn. Syst., suppl. (2003), 327-335.

  \bibitem{GW} Z. Guo,  J. Wei,  Global solution branch and Morse index estimates of a semilinear
  elliptic equation with super-critical exponent, Trans. Amer. Math. Soc., \textbf{363},  (2011) n. 9, 4777-4799.

 \bibitem{JSell}
  Johnson R., Concerning a theorem of Sell,  J. Differential Equations, {\bf 30}
  (1978), 324-339.

\bibitem{JPY2}
  Johnson R., Pan X.B., Yi Y.F.,
Singular ground states of  semilinear elliptic equations via invariant manifold theory,
  Nonlinear Anal.,  \textbf{20} (1993),  1279-1302.


\bibitem{JPY}
Johnson R., Pan  X.B., Yi  Y.F., The Melnikov method and elliptic equation with critical exponent,
Indiana Univ. Math. J., \textbf{43} (1994), 1045-1077.


\bibitem{JK}
Jones C., K\"upper T., On the infinitely many solutions of a semilinear elliptic equation,
SIAM J. Math. Anal., \textbf{17}  (1986), 803-835.

\bibitem{KabYY}
Kabeya Y.,   Yanagida E.,  Yotsutani S.,
Existence of nodal fast-decay solutions to ${\rm div}(|\nabla u|^{m-2}\nabla u)+K(|x|)|u|^{q-1}u=0$ in $\mathbb{R}^n$,
Differential Integral Equations, \textbf{9} (1996),  981-1004.


\bibitem{KNY}
Kawano N., Ni W.M., Yotsutani S.,
A generalized Pohozaev identity and its applications,
J. Math. Soc. Japan, {\bf 42} (1990),  541-564.


\bibitem{KSY}
Kawano N., Satsuma J., Yotsutani S.,
Existence of positive entire solutions of an Emden-type elliptic equation,
Funkcial. Ekvac., {\bf 31} (1988), 121–145.


\bibitem{KN}
Kusano T., Naito M., Oscillation theory of entire solutions of second order superlinear elliptic equations,
Funkcial. Ekvac., {\bf 30} (1987),  269-282.

\bibitem{Lsolo}
Li  Y.Y., Prescribing scalar curvature on $\mathbb{S}^n$ and related problems, Part II: Existence and compactness,
Comm. Pure Appl. Math., {\bf 49} (1996), 541-597.


\bibitem{LNW}
Li Y.Y., Nguyen L., Wang B.,
The axisymmetric $\sigma_k$-Nirenberg problem, J. Funct. Anal., {\bf 281} (2021),  Paper N. 109198, 60 pp.

\bibitem{LL}
Lin C.S., Lin S.S.,
Positive radial solutions for $\Delta u+K u^{\frac{n+2}{n-2}} =0$ in $\R^n$  and related topics,
Appl. Anal., {\bf   38}  (1990),   121-159.


  \bibitem{MP} F. Merle,  L.A. Peletier,  Positive solutions of elliptic equations involving supercritical growth, Proc. Roy.
 Soc. Edinburgh sect A, \textbf{118}, (1991), 49-62.

\bibitem{Ni}
Ni W.M.,
On the elliptic equation $\Delta u+K(x)  u^{\frac{n+2}{n-2}} =0$, its generalizations, and applications in geometry,
Indiana Univ. Math. J., {\bf 31} (1982), 493-529.

\bibitem{NiSe}
Ni W.M., Serrin J.,
Nonexistence theorems for quasilinear partial differential equations,
Rend. Circ. Mat. Palermo (2), {\bf 8}  (1985), 171-185.



\bibitem{Sharaf}
Sharaf K., A note on a second order PDE with critical nonlinearity,
Electron. J. Qual. Theory Differ. Equ., {\bf 10} (2019), 16 pp.



\bibitem{YY93} Yanagida E.,  Yotsutani S.,
 Classification of the structure of positive radial solutions to  $\Delta u +K(|x|) u^p=0$ in $\mathbb{R}^n$,
 Arch. Rational Mech. Anal., {\bf 124} (1993), 239-259.




\bibitem{YY94}
Yanagida E.,  Yotsutani S., Existence of nodal fast-decay solutions to
 $\Delta u +K(|x|) |u|^{p-1}u=0$ in $\mathbb{R}^n$, Nonlinear Anal., {\bf 22} (1994), 1005-1015.





\end{thebibliography}
\end{document}